\documentclass[reqno]{amsart}
\makeatletter
\@namedef{subjclassname@2020}{%
  \textup{2020} Mathematics Subject Classification}
\makeatother
\usepackage{amssymb}
\usepackage[utf8]{inputenc} 
\usepackage[T1]{fontenc}    
\usepackage{url}            
\usepackage{booktabs}       
\usepackage{amsfonts}       
\usepackage{nicefrac}       
\usepackage{microtype}      
\usepackage{color}
\usepackage{lipsum}
\usepackage{float}
\usepackage{calrsfs}
\usepackage{enumerate}
\usepackage{tikz-cd}
\usepackage{tikz}
\usepackage{graphicx}
\usepackage{old-arrows} 
\usetikzlibrary{arrows}
\usepackage{comment}
\usepackage{eqnarray}
\usepackage{amsmath}
\usepackage{marginnote}
\usepackage{mathtools}
\usepackage{adjustbox}
\DeclareMathAlphabet{\pazocal}{OMS}{zplm}{m}{n}
\usepackage{hyperref}
\usepackage{graphicx}
\usetikzlibrary{graphs, automata, positioning}
\usepackage{bm}

\DeclareMathAlphabet\mathbfcal{OMS}{cmsy}{b}{n}

\input xygraph
\input xy
\xyoption{all}
\def\path{\mathop{\hbox{\rm Path}}}

\def\obj{\mathop{\text{\rm Obj}}}
\def\L#1{\displaystyle\lim_{\leftarrow}#1}
\def\Z{\mathbb{Z}}

\def\complex{\mathbb{C}}
\def\N{\mathbb{N}}

\def\grp{\mathop{\text{\bf Grp}}}
\def\set{\mathop{\text{\bf Set}}}

\def\F{\mathfrak F}

\def\CC{\mathfrak C}

\def\U{\mathcal U}

\def\aut{\mathop{\text{Aut}}}

\def\a{\alpha}

\def\B{\mathcal{B}}
\def\l{\lambda}
\def\m{\mu}
\def\t{\tau}
\def\s{\sigma}
\def\w{\omega}
\def\vec{\mathbf{v}}
\def\wec{\mathbf{w}}
\def\Q{\mathbb{Q}}
\def\A{\mathcal{A}}
\newcommand{\LI}{{\rm 2LI}}

\newcommand{\E}{\mathbfcal{E}}
\def\remove#1{}

\def\car{\mathop{\text{\rm char}}}
\def\cat{\text{\bf Cat}}
\def\I{\mathfrak{I}}
\def\mor{\mathop{\text{\rm Mor}}}

\def\C{\mathcal{C}}

\def\iu{\hbox{\bf i}}
\DeclareMathOperator{\Span}{span}

\def\graph{\text{\bf Grph}}
\def\diag{\text{\rm Diag}}

\newtheorem{lemma}{Lemma}[section]
\newtheorem{corollary}[lemma]{Corollary}
\newtheorem{theorem}[lemma]{Theorem}
\newtheorem{proposition}[lemma]{Proposition}

\theoremstyle{remark}
\newtheorem{observation}[lemma]{Observation}

\theoremstyle{definition}
\newtheorem{definition}[lemma]{Definition}

\newtheorem{example}[lemma]{Example}

\newtheorem{notation}[lemma]{Notation}
\newtheorem{remark}[lemma]{Remark}

\usepackage{array}
\newcolumntype{R}[1]{>{\arraybackslash$}p{#1}<{$}}

\title{Solenoids in automorphism groups of evolution algebras}
\author[Y. Cabrera]{Yolanda Cabrera Casado}
\address{Y. Cabrera Casado: Departamento de Matem\'atica Aplicada, E.T.S. Ingenier\'\i a Inform\'atica, Universidad de M\'alaga, Campus de Teatinos s/n. 29071 M\'alaga.   Spain.}
\email{yolandacc@uma.es}

\author[M. I. Gon\c calves]{Maria Inez Cardoso Gon\c calves}
\address{M. I. Cardoso Gon\c calves: Departamento de Matem\'atica, Universidade Federal de Santa Catarina, Florian\'opolis, SC, 88040-900 - Brazil}
\email{maria.inez@ufsc.br}

\author[D. Gon\c calves]{Daniel Gon\c calves}
\address{D. Gon\c calves: Departamento de Matem\'atica, Universidade Federal de Santa Catarina, Florian\'opolis, SC, 88040-900 - Brazil}
\email{daemig@gmail.com}

\author[D. Mart\'\i n]{Dolores Mart\'\i n Barquero}
\address{D. Mart\'\i n Barquero: Departamento de Matem\'atica Aplicada, Escuela de Ingenier\'\i as Industriales, Universidad de M\'alaga, Campus de Teatinos s/n. 29071 M\'alaga.   Spain.}
\email{dmartin@uma.es}

\author[C. Mart\'\i n]{C\'andido Mart\'\i n Gonz\'alez}
\address{C. Mart\'\i n Gonz\'alez: Departamento de \'Algebra Geometr\'{\i}a y Topolog\'{\i}a, Fa\-cultad de Ciencias, Universidad de M\'alaga, Campus de Teatinos s/n. 29071 M\'alaga.   Spain.}
\email{candido\_m@uma.es}

\author[I. Ruiz]{Iv\'an Ruiz Campos}
\address{I. Ruiz Campos:  Departamento de \'Algebra Geometr\'{\i}a y Topolog\'{\i}a, Fa\-cultad de Ciencias, Universidad de M\'alaga, Campus de Teatinos s/n. 29071 M\'alaga. Spain.}
\email{ivaruicam@uma.es}

\begin{document}

\begin{abstract}

Let $A$ be an evolution algebra (possibly infinite-dimensional)  equipped with a fixed natural basis $B$, and let $E$ be the associated graph defined by Elduque and Labra. 
 We describe the group of automorphisms of $A$ that are diagonalizable with respect to $B$. This group arises as the inverse limit of a functor (a diagram) from the category associated with the graph $E$ to the category of groups. In certain cases, this group can be realized as a dyadic solenoid. Additionally, we investigate the automorphisms that permute (and possibly scale) the elements of $B$. In particular, for algebras satisfying the 2LI condition, we provide a complete description of their automorphism group.
\end{abstract}
\maketitle
\section{Introduction}

One way to think about the mathematics underlying evolution algebras is through the formalization of asexual reproduction. In sexual reproduction, two organisms combine to produce a third one, whereas in asexual reproduction, a single organism produces offspring genetically similar to the parent. 
Thus, sexual reproduction can be schematized as an interaction in which organisms \(x\) and \(y\) couple to produce \(z\).
This process can be thought of as a multiplication \(xy = z\). 
On the other hand, asexual reproduction involves only one parent. So, instead of two parents (\(xy\)), we have a single parent squared (\(x^2\)). Since the offspring is genetically similar to \(x\), the resulting equation is \(x^2 = x\).

\begin{figure}[h]
    \centering
  \[
\begin{tikzcd}[column sep=small, row sep=small]
x \arrow[rd, no head] &                   & y \arrow[ld, no head] \\
                      &  \arrow[d] &                       \\
                      & z                 &
\end{tikzcd}
\hspace{1cm}
\begin{tikzcd}
  x^2 \arrow[d] \\
  x
\end{tikzcd}
\]

    \caption{On the left side we have the scheme of sexual reproduction and on the right side the asexual one.}
    \label{fig_esquemasrep}
\end{figure}
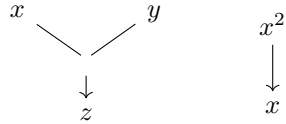

Now, if we embed these ideas in an algebra setting, sexual reproduction gives us equations of the form
$e_ie_j=\sum_{k}\omega_{ijk}e_k$, where the $e_i$'s represent the different types of parent organisms,
and the scalars $\omega_{ijk}$ correspond to probabilities. 
In contrast, asexual reproduction leads to equations of the form $e_i^2=e_i$, which characterize a well-known class of evolution algebras. However, if we take into account the possibilities of mutations, errors, or perturbations from the ideal model (described by $e_i^2=e_i$), we obtain a more realistic model given by
$e_i^2=\sum_k \omega_{ki}e_k$,  where, in principle, the scalar $\omega_{ii}$ is close to $1$ while the others are close to $0$. 
If the deviation from the ideal model becomes significant, we arrive at the concept of a general evolution algebra, where there are no constraints requiring
$\omega_{ii}$ to be close to $1$ and the others to be close to $0$. Notice that within the context of asexual reproduction,  the product $e_ie_j$ for distinct $i$ and $j$ should be zero, as no sexual reproduction occurs in this scheme.

It turns out that evolution algebras are non-associative and fundamentally differ from traditional algebraic structures: not only does each basis element square to a linear combination of other basis elements, but the product of any two distinct basis elements is zero. This unique property makes evolution algebras particularly useful for studying dynamical systems, Markov processes, and population genetics.

The first comprehensive reference on evolution algebras is the book by Jianjun Paul Tian \cite{Tianbook}, which introduces the theory, explores its algebraic properties, and discusses classification and applications in biology and physics. The book presents key concepts of the theory and demonstrates their relevance in various mathematical and scientific contexts.

Following \cite{Tianbook}, numerous contributions to the theory of evolution algebras have appeared in the literature,  for example \cite{Elduque2}, \cite{casas}, and  \cite{csv1}. Evolution algebras exhibit deep connections with graph theory and stochastic processes, making them a versatile interdisciplinary tool. They offer valuable insights into non-Mendelian inheritance mechanisms and find applications in mathematical ecology, epidemiology, and theoretical physics. See \cite{ceballos} for a comprehensive background on the state of the art of this type of non-associative algebra.
The interactions between evolution algebras and graph theory, as established in \cite{Elduque2}, allow the description of the affine group scheme of automorphisms of finite-dimensional perfect evolution algebras, as presented in \cite{Elduque}. This result directly precedes our current work. Additionally, in \cite{Hopf}, the authors determine the Hopf algebras that represent the affine group scheme of automorphisms of two-dimensional evolution algebras and explore their relation to universal associative and commutative representations of these algebras. 

The solenoid is a classic object in topology and dynamics, first introduced by Vietoris and later studied more systematically by van Dantzig and others. While the term solenoid is sometimes used to describe various kinds of topological or algebraic structures, one of the most well-studied and structurally rich is the dyadic solenoid, a one-dimensional compact connected abelian group that arises as the inverse limit of circles under degree-two covering maps.  More generally, a solenoid is a compact connected topological space (i.e., a continuum) that may be obtained as the inverse limit of an inverse system of topological groups and continuous homomorphisms. They naturally arise as minimal sets in certain dynamical systems, especially as attractors in flows on 3-manifolds or in suspension flows over symbolic systems. 

This work describes the automorphism group of certain evolution algebras as solenoids. This suggests that the symmetries and internal structure of the algebra are far from trivial, showing a kind of nested, self-similar symmetry, as solenoids are constructed from repeated circle covers.

We begin by investigating diagonalizable automorphisms—those that diagonalize relative to a fixed natural basis—without assuming finite-dimensionality or perfection. 
We find that the group of such automorphisms (once a natural basis is chosen) can be characterized as the inverse limit of a diagram (see Definition \ref{notorius}) arising from the graph associated to the evolution algebra in question. 
In the second part of our work, we focus on automorphisms that both permute and scale the elements of the chosen natural basis.

The paper is organized as follows. We begin with a preliminary section, where we recall relevant concepts and set up notation concerning evolution algebras, graphs and their associated categories, inverse limits, and solenoids. In Section~\ref{sec_diagonaliable_automorphism}, we study diagonalizable automorphisms of evolution algebras, i.e., those that scale the basis elements. { In Example~\ref{diomucho}, we identify the automorphisms with Tate modules.}
Moreover, we present examples where the group of diagonalizable automorphisms coincides with a solenoid and prove a general theorem characterizing this group as the inverse limit of a certain functor (see Theorem~\ref{arop}).

Section~\ref{sec_deviation_finitedimensional} focuses on the contrasting behavior of automorphism groups in finite versus infinite-dimensional evolution algebras. For perfect evolution algebras, we explore properties such as nondegeneracy, finiteness of the automorphism group, and invertibility of the structure matrix.

Before addressing non-diagonalizable automorphisms, Section~\ref{sec_2LI} presents an interlude in which we study evolution algebras satisfying the 2LI Condition. In particular, we describe the relationships between natural bases of such algebras (see Theorem~\ref{2LI}).

We conclude in Section~\ref{sec_non_diagonalizable} with a study of non-diagonalizable automorphisms. We introduce weighted graphs and show that weighted graphs satisfying Condition~(Sing) correspond bijectively to pairs $(A,B)$, where $A$ is an evolution algebra and $B$ is a natural basis. This correspondence is category theoretic and allows us to analyze the automorphisms of an evolution algebra with a fixed basis as automorphisms of a weighted graph (see Proposition~\ref{weenoh}). We then show that an automorphism of a weighted graph induces an inverse limit, which can be naturally embedded into the automorphism group of $A$. In Theorem~\ref{pescaito}, we describe the union of all such automorphisms as a semidirect product of the diagonalizable automorphisms of $A$ (with respect to a basis $B$) and the automorphisms of the associated weighted graph. We end the paper by presenting methods for computing inverse limits arising from automorphisms of evolution algebras.

\section{Preliminaries}\label{sec_preliminaries}
In what follows, we will denote the natural numbers without zero by $\N^*$. We will use the notation $C_n=\Z/n\Z$ for the cyclic group of order $n\ge 1$. This will prevent confusion with the group
of $p$-adic integers $\Z_p$ defined as the inverse limit of the system $\{C_{p^n}\}_{n\ge 1}$ with connecting homomorphism induced by $C_{p^{n+1}}\to C_{p^n}$ such that $\bar k\mapsto \bar j$, where
$j$ is the remainder of the division of $k$ by $p^n$.

An algebra $A$ over a field $K$ is considered an \emph{evolution algebra} if there exists a basis $B=\{e_i\}_{i\in \Lambda}$ such that $e_ie_j=0$ for every $i, j \in \Lambda$ with $i\neq j$. Such a basis is called a \emph{natural basis}. Denote by $M_B=(\omega_{ij})$ the \emph{structure matrix} of $A$ relative to $B$, where $e_i^2 = \sum_{j\in \Lambda} \omega_{ji}e_j$.  

A {\it directed graph} is a quadruple $E=(E^0, E^1,s,r)$ where $E^0, E^1$ are sets and $s,r\colon E^1\to E^0$ are maps (called source and range, respectively). We will use the terms graph and directed graph interchangeably.  The elements of $E^0$ are called \emph{vertices}, and the elements of 
$E^1$ are called \emph{edges} of $E$. A \emph{path} $\mu$ of length $m$ is a finite chain of edges $\mu=f_1\ldots f_m$ such that $r(f_i)=s(f_{i+1})$ for $i=1,\ldots,m-1$. The vertices will be considered trivial paths, namely, paths of length zero. If $f\in E^1$ is such that $r(f)=s(f)$, then we say that $f$ is a loop and the vertex $r(f)$ is the basis of the loop.
If $S\subset E^0$, then denote by $T(S)$ the \emph{tree} of $S$ where $$T(S)=\{v\in E^0 \colon \text{ exist }\lambda \in \text{Path}(E) \text{ and } u \in S \text{ with } s(\lambda)=u, r(\lambda)=v\}.$$

A graph $E$ satisfies Condition~(Sing) if for every two vertices $u,v \in E^0$, we have $|s^{-1}(u) \cap r^{-1}(v)|\leq 1$.

If $A$ is an evolution algebra with natural basis $B=\{e_i\}_{i\in \Lambda}$ and structure matrix $M_B=(\omega_{ij})$, then we denote by $E=(E^0, E^1, r_E, s_E)$ the \emph{directed graph associated to $A$ relative to $B$}, which is defined by setting $E^0=\{e_i\}_{i \in \Lambda}$ and 
drawing an edge from $e_j$ to $e_i$ if and only if $\w_{ij}\ne 0$, see \cite{Elduque2}. 
Following \cite{maclane}, 
For a category $\C$, we will denote the classes of objects and morphisms by $\obj(\C)$ and $\mor(\C)$, respectively. 
Next, we define several categories that we will use. By $\cat$ we mean the category of small categories and functors, and by $\grp$ we denote the category of groups.

The category of graphs $\graph$ is defined in \cite[II, sect. 7, p. 48]{maclane} (in the terminology of \cite{maclane}, our graphs are termed "small graphs").

Let $E$ be a graph. We recall the construction of the free category generated by the graph $E$, as in \cite{maclane}:  $\I_E$ is a small category such that $\obj(\I_E)=E^0$ and $\mor(\I_E)=\cup_{u,v\in E^0} \hom_{\I_E}(u,v)$, where for $u,v\in E^0$ we define $\hom_{\I_E}(u,v)$ as the set of all paths with source $u$ and range $v$.  Thus, we have a functor $\I\colon\graph\to\cat$ such that $E\mapsto \I_E$. On the other hand, there is
a forgetful functor $U\colon\cat\to\graph$ given by $U(\C)=(\obj(\C),\mor(\C),s,r)$, where for any  morphism $f\in\mor(\C)$ with  $f\in\hom_\C(X,Y)$ we put $s(f)=X$ and $r(f)=Y$. As shown in \cite[II, sect. formula (6), p.51]{maclane} we can say that the functor $\I$ is the left adjoint of the forgetful functor $U$. This implies that there is a bijection 
\begin{equation}\label{noitcnujda}
\hom_\graph(E,U(\C))\cong\hom_\cat(\I_E,\C)
\end{equation}
for any graph $E$ and small category $\C$. Note that there may exist morphisms in $\I_E$ that are not edges of the graph $E$ (for instance, a path of length greater than one). The main use of \eqref{noitcnujda} is that 
one can define a functor $\I_E\to\C$ simply by giving a graph homomorphism $E\to U(\C)$. We will use this fact in the sequel without further warning. In conclusion, one can associate a category to any graph. Hence, we can associate a category to each pair formed by an evolution algebra and a natural basis (via its associated graph).

\begin{definition}   \label{tonta}
We define the small category $\C_0$ as the category whose unique object is $\obj(\C_0)=\{K^{\times}\}$ and $\mor(\C_0)=\{\varphi:K^\times \to K^\times \vert \,  \varphi \,\text{ is a group homomorphism}\}$. 
\end{definition}

\begin{remark}   \label{atnot} 
Notice that $\C_0$ is a subcategory of $\grp$. So any
functor from any category $\mathcal{A}$ to $\C_0$, induces a functor 
$\mathcal{A}\to\grp$.
\end{remark}

 \begin{definition}\label{def_cone}
     Consider a small category $I$ and $\F\colon I\to\C$ a functor with values in a category $\C$. Recall that a {\it  cone} for the functor $\F$ is
an object $G$ in $\C$ and a collection $\{t_i\}_{i\in I}$ of homomorphisms in $\C$ with  
$t_i\colon G\to\F(i)$ such that the triangles
\begin{center}
\begin{tikzcd}[cramped, column sep=small]
G \arrow{r}{t_i}  \arrow[rd,"t_j"'] 
  & \F(i) \arrow{d}{\F(a)} & i\arrow[d,"a"]\\
    & \F(j) & j
\end{tikzcd}
\end{center}
commute for any arrow $a\colon i\to j$ in $I$. 
 \end{definition}
\begin{observation}
The set $\{\F(i)\}_{i \in I}$ can be thought as a family of objects indexed by $I$ and, for each arrow $a \colon i\to j$ in $I$, the
morphism $\F(a) \colon \F(i)\to\F(j)$ satisfies $\F(a) \circ t_i = t_j$. 
\end{observation}
 
Now, we can consider the
\emph{category of cones for a fixed functor $\F$}: its objects are the cones for $\F$ and if $(G,\{t_i\}_{i\in I})$ and $(G',\{t'_i\}_{i\in I})$  are cones for $\F$, then a homomorphism from $(G,\{t_i\}_{i\in I})$ to $(G',\{t'_i\}_{i\in I})$ is a homomorphism $\varphi\colon G\to G'$ making commutative the triangles 
\begin{center}
\begin{tikzcd}[cramped, column sep=tiny]
G \arrow[dr,"t_i"'] \arrow[rr,"\varphi"] & & G' \arrow[dl,"t'_i"]\\
& \F(i) & 
\end{tikzcd}
\end{center}
for any $i\in I$. 

\begin{definition}\label{def_inverse_limit}
A \emph{inverse limit of the functor} $\F$
(denoted $\displaystyle\lim_{\leftarrow}\F$) It is a terminal object in the category of cones for $\F$. This means that $\L{\F}$ is an object in $\C$ endowed
with homomorphisms $t_i\colon\L \F\to\F(i)$ for any $i\in I$, such that
\begin{enumerate}
    \item $\F(a)t_i=t_j$ for any $a\colon i\to j$ in $I$.
    \item For any other $H$ endowed with homomorphisms $s_i\colon H\to \F(i)$ satisfying $\F(a)s_i=s_j$ when $a\colon i\to j$ in $I$, there is a unique homomorphism $\theta\colon H\to \L{\F}$ such that $t_i \circ \theta=s_i$ for any $i$.
\end{enumerate}
\begin{center}
\begin{tikzcd}[row sep=tiny]
\L{\F}\arrow[dr,"t_i"] \arrow[drr,controls={+(1.5,0.5) and +(-1,0.8)}, "t_j"]&       &     \\
   & \F(i)\arrow{r}{\F(a)} & \F(j).\\
H \arrow[ur,"s_i"] \arrow[uu,dashed,"\theta"] \arrow[urr,controls={+(1.5,-0.5) and +(-0.5,-0.5) }, "s_j"']&       &   
\end{tikzcd}
\end{center}
\emph{Direct limits} are defined dually by using \emph{cocones}.
\end{definition}

See \cite[Sect. 4, \S III]{maclane} for a general reference on categories and related notions.

\begin{remark}\label{remark_lolitienesueño}
The inverse limit $\L{\F}$ for a functor $\F\colon I\to\set$ can be constructed as the set of all $(x_i)_{i\in I}\in\prod_{i\in I}\F(i)$ such that $\F(a)(x_i)=x_j$, whenever $a\colon i\to j$ in $I$. The map $t_j \colon \L \F \to \F(j)$ is given by $t_j((x_i)_{i\in I}) = x_j$. For a functor $\F\colon I\to\grp$, the inverse limit can be constructed analogously.
\end{remark}

An interesting example is the so-called dyadic solenoid. This is a construction that appears naturally in our study of automorphisms of
evolution algebras.

\begin{example}\label{example_solenoide}
    Consider the category $\I_E$, where $E$ is the graph whose vertices are the natural numbers $0,1,\ldots$ and the only edges are $i+1\to i$, see Figure \ref{fig_example_solenoide}.
    \begin{figure}[h]
        \begin{equation*}\cdots\to \overset{i+1}{\bullet}\to \overset{i}{\bullet}\cdots \overset{1}{\bullet}\to \overset{0}{\bullet}\end{equation*}
        \caption{Graph $E$ of example \ref{example_solenoide}.}
        \label{fig_example_solenoide}
    \end{figure}

Let ${\bf n}=(n_1,n_2,\ldots)$ be a sequence of positive integers and consider the functor $\F\colon \I_E\to\grp$ such that $\F(i)=S^1$ the unit circle and $\F(i)\to \F(i-1)$ is the map $x\mapsto x^{n_i}$. This defines an inverse system of groups 
\begin{equation}\label{box}
\cdots\to S^1\buildrel{f_3}\over{\to} S^1 \buildrel{f_2}\over{\to} S^1\buildrel{f_1}\over{\to} S^1,
\end{equation}
where we have abbreviated $f_i:x\mapsto x^{n_i}$. Then
we define $\Sigma_{\bf n}:=\L{\F} $ to be the inverse limit (as a topological group). This topological space $\Sigma_{\bf n}$ is called the \emph{${\bf n}$-adic solenoid}. For the special case ${\bf n}=(2,2,\ldots)$ we get the so-called
\emph{dyadic solenoid}.
\end{example}

Dyadic solenoids were first introduced in \cite{V} and in \cite{D} for all constant sequences. Solenoids are compact connected topological spaces. However, they are not locally connected or path connected. Of course, they are abelian groups and are examples of what is known as a \emph{protorus} (compact connected topological abelian group). Recalling the definition of inverse limit, a materialization of $\Sigma_{\bf n}$ for ${\bf n}=(2,2,\ldots)$, is the group of all
sequences $(x_0,x_1,\ldots )$ with $x_i\in S^1$ and $x_{i+1}^2=x_i$ for all $i$ (with the componentwise product). As it is well known, solenoids are not Lie groups since connected Lie groups are path-connected.
Instead of using the Lie group $S^1$, we can consider the algebraic group $K^\times$, the graph $E$ above, and ${\bf n}=(n_1, n_2,\ldots)$ fixed. Then,
we can define the functor $\F\colon \I_E\to\grp$ such that $\F(i)=K^\times$ and the map $\F(i)\to\F(i-1)$ is $x\mapsto x^{n_i}$. This defines an inverse system 
\begin{equation}\label{box2}
\cdots\to K^\times\buildrel{f_3}\over{\to} K^\times \buildrel{f_2}\over{\to} K^\times\buildrel{f_1}\over{\to} K^\times,
\end{equation}
of groups and the group  
$\L{\F}$ is what we term as a \emph{generalized solenoid}.

\begin{figure}[h]
    \centering
    \includegraphics[width=0.4\linewidth]{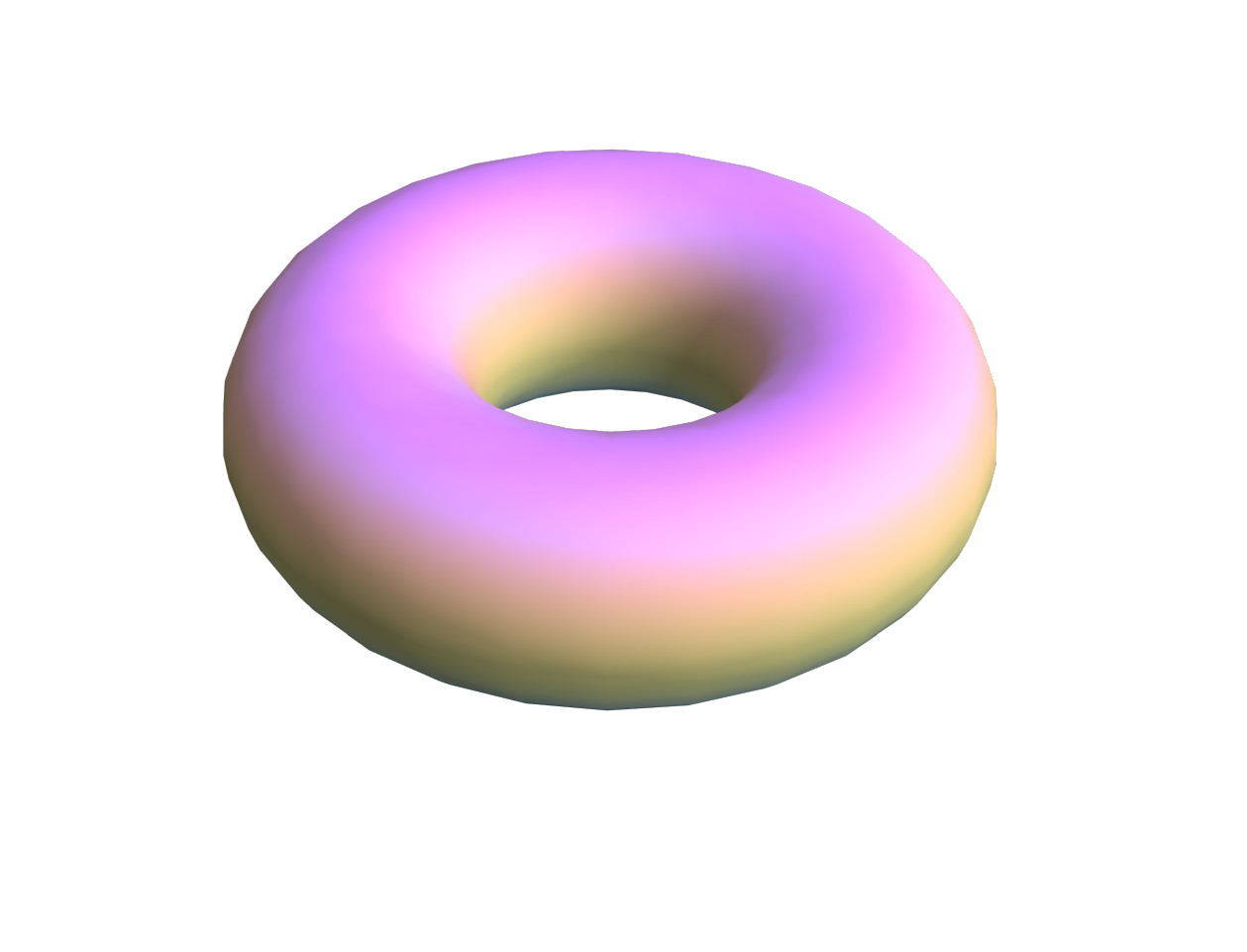}
    \caption{\small Drawn with \cite{Mathematica}.}
    \label{fig_donut}
\end{figure}

\begin{figure}[h]
    \centering
    \includegraphics[width=0.4\linewidth]{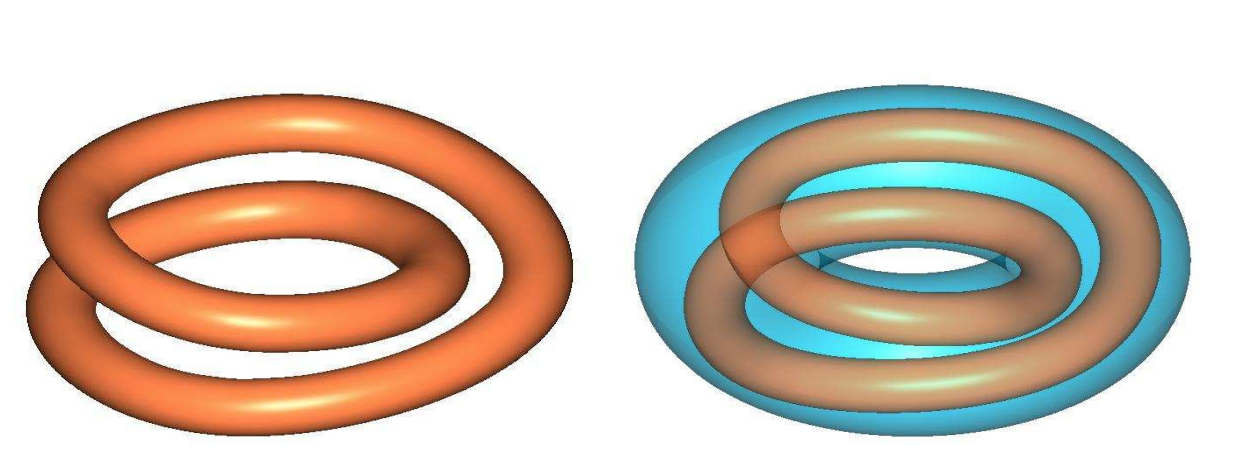}
    
    \includegraphics[width=0.4\linewidth]{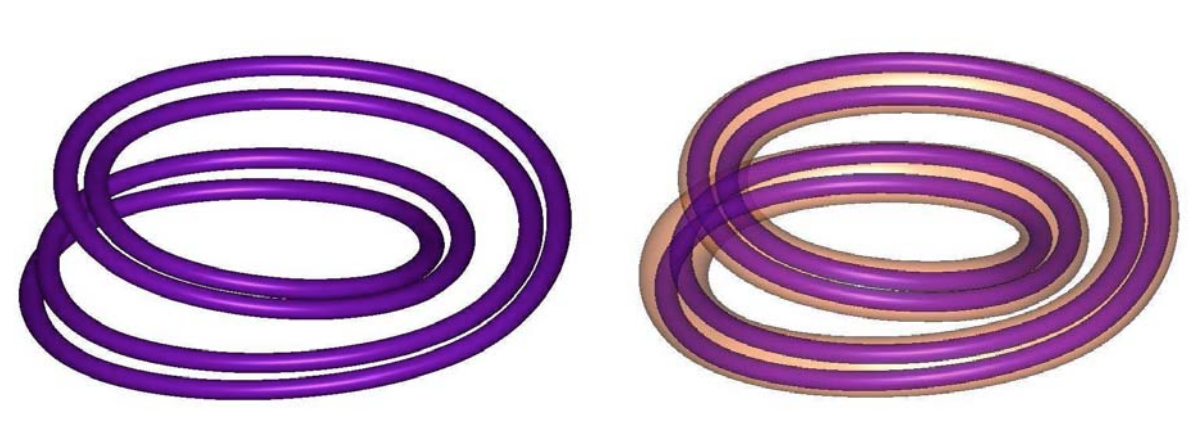}
    \caption{\small See \cite{selenoide}.}
    \label{fig_solenoide}
\end{figure}

 To end this section, we define the set $\mu_n^{(a)}(R):=\{x\in R^\times : x^n=a\}$ for any $a\in R^\times$, $n \in \N^*$ and $R$ a commutative associative unitary $K$-algebra. Notice that $\mu_n(R):=\mu_n^{(1)}(R)$ is the group of $n$-th roots of the unit. In particular, we have that $\mu_2(K)=\{\pm 1\}\cong C_2$ if  $\car(K)\neq 2$ and $\mu_2(K)=\{ 1\}$ if $\car(K)=2$. There is an action (which is transitive and free ) $\mu_{2^n}(R)\times\mu_{2^n}^{(a)}(R)\to\mu_{2^n}^{(a)}(R)$, given by multiplication. We also define the sets
$$\mu_{2^\infty}^{(a)}(R):=\{(x_i)_{i\ge 1}\colon x_i\in\mu_{2^i}^{(a)}(R), x_{i+1}^2=x_i\}$$
for any $a \in R^\times$. In particular, if $a=1$, we write $\mu_{2^\infty}(R):=\mu_{2^\infty}^{(1)}(R)$ which is a group and there is an action $\mu_{2^\infty}(R)\times \mu_{2^\infty}^{(a)}(R)\to \mu_{2^\infty}^{(a)}(R)$ given by componentwise multiplication. Notice that since each $\mu_{2^\infty}^{(a)}(R)$ is contained in the group $(R^\times)^{\N^*}$ we can consider the union
$$\bigcup_{a\in K^\times} \mu_{2^\infty}^{(a)}(R)$$
And this is a group.

\section{Diagonalizable automorphisms of evolution algebras}\label{sec_diagonaliable_automorphism}

\begin{definition}\label{def_diag}
    Assume that $A$ is an evolution algebra and $B=\{u_i\}_{i\in\Lambda}$ a natural basis of $A$.
An automorphism $f\in\aut(A)$ is said to be {\it diagonalizable} (relative to $B$) if $f(u_i)\in K^\times u_i$ for any $i\in\Lambda$. We define {\it the group of diagonalizable automorphisms relative to $B$}, denoted by $\diag(A;B)$, as 
\begin{equation*}
    \diag(A;B):=\{f\in\aut(A)\colon f(u_i)\in K^\times u_i, i\in\Lambda\}.
\end{equation*}
\end{definition}

\begin{remark}
Notice that, as the notation suggests, the group of diagonalizable automorphisms strongly depends on the chosen basis (unless the algebra satisfies additional properties). For example, consider the evolution algebra $A$ with natural basis $B=\{e_1,e_2\}$ and product defined by $e_1^2=e_1+e_2$ and $e_2^2=0$. In this case, $\diag(A;B)=\{1\}$. However, if we consider the algebra with the natural basis $B'=\{e_1+e_2,e_2\}$, then $\diag(A;B')=K^{\times}$. 
\end{remark}

\begin{example}
    Take the graph
$E$ such that $E^0:=\{w\}\sqcup\{v_i\}_{i\in\in\Z}$ and $E^1:=\{f_i\}_{i\in\Z}$ with $s(f_i)=v_i$ and $r(f_i)=w$ for any $i$.
This is the graph associated to the evolution algebra $A$ with natural basis $B=E^0$ and multiplication $w^2=0$, $v_i^2=w$ for any $i$. 
\begin{center}
    \begin{tikzcd}[column sep=small, row sep=large]
\arrow[r, dotted,no head]& v_{-i-1}\arrow[r, dotted, no head]\arrow[drrr] &v_{-i}\arrow[r, dotted, no head]\arrow[drr] & v_{-1} \arrow[dr] & 
v_0 \arrow[d] & v_1 \arrow[dl]\arrow[r,dotted,no head] & v_i \arrow[r, no head,dotted]\arrow[dll] & v_{i+1}\arrow[r, no head, dotted] \arrow[dlll]& \ \\
& & & & w
\end{tikzcd}
\end{center} 

If $f\in\diag(A;B)$ then there are nonzero scalars $x_i$ such
that $f(v_i)=x_i v_i$ ($i\ne 0$) and $f(w)=x_0 w$. Then 
$f(v_i^2)=f(w)=x_0 w$ while $f(v_i)^2=(x_i v_i)^2=x_i^2 w$, whence
$x_0=x_i^2$ for any $i\ne 0$. Assuming that the ground field $K$ is
quadratically closed, $\sqrt{x_0}$ exists and $x_i=\pm x_j$ if $i\ne j$. We analyse two cases:
\begin{enumerate}
    \item If $\car(K)=2$ then $\l:=\sqrt{x_0}$ exists and it is unique. So $f$ is of the form $f(w)=\l^2 w$, $f(v_i)=\l v_i$ ($i\ne 0$), where $\l\in K^\times$. In this case, $\diag(A; B)\cong K^\times$, a one-dimensional torus.
    \item If $\car(K)\ne 2$ we take $\l$ to be one of the square roots of $x_0$. Then $f(w)=\l^2 w$ as before, but $f(v_i)=\l\epsilon_i v_i$ where $\epsilon_i=\pm 1$ for $i\ne 0$.
    In this case $\diag(A;B)\cong K^\times\times (S^{0})^\Z$, where $(S^{0})^\Z$ is the multiplicative group of sequences $(\epsilon_i)_{i\in\Z}$ with $\epsilon_i\in S^0=\{\pm 1\}$. 
\end{enumerate}
Notice that in any case, other automorphisms of this algebra are 
possible: if we take a bijection $\s\colon\Z\to\Z$, then the map
$g\colon A\to A$ fixing $w$ and making $g(v_i)=v_{\s(i)}$ is an automorphism of $A$.
\end{example}

If there is essentially a unique natural basis (up to scaling and reordering), the notation $\diag(A; B)$ can be simplified to $\diag(A)$. For example, when $A$ is a finite-dimensional perfect evolution algebra (see \cite{Elduque}). In section \ref{sec_2LI} and subsection \ref{subsec_inversibility}, we will devote some attention to the problem of essential uniqueness of natural bases. 

\subsection{Interactions with evolution algebras}
In this subsection, we illustrate with some motivating examples how the solenoids presented above coincide
with the diagonalizable automorphisms of evolution algebras. 

\begin{example}  
Let $\I_E$ be the category associated to the following graph:

\begin{equation}\label{diaone}
E: \hspace{1cm}
\begin{tikzcd}[row sep=normal, column sep=normal]
5 \arrow[rd,controls={+(0,-0.5) and +(-0.3,0)},"a_{51}"']  &         &         \\
        & 1\arrow[r,"a_{12}"] \arrow[lu,controls={+(.5,.5) and +(0.5,0.1)},"a_{15}"']&  2\arrow[d,"a_{23}"]\\
        & 4  \arrow[u,"a_{41}"]&  3\arrow[l,"a_{34}"]
\end{tikzcd}
\end{equation}

Consider now the functor $\F\colon \I_E\to\C_0$,  where $\C_0$ is defined in  \ref{tonta} and $K$ is a field with cubic roots of the unit $\{1, \omega, \omega^2\}$, induced by the graph morphism $E \to U(\C_0)$ (see the adjunction defined in formula \eqref{noitcnujda}) such that $i \mapsto K^\times$ with $a_{ij} \mapsto s$ where $s \colon K^\times \to K^\times$ such that $s(x) = x^2$.
By Remark \ref{atnot} we may consider $\F$ as a functor $\F\colon\I_E\to\grp$. Then, $\L{\F}\cong\mu_3(K)$.

 To prove this, we apply the functor $\F$ to the category $\I_E$ and obtain a diagram so that each time an element advances through a (non-identity) arrow, we square it. There are group homomorphisms $t_i\colon\L{\F}\to K^\times=\F(i)$ such that 
$t_i(x)^2=t_j(x)$ for each $\a\colon\F(i)\to\F(j)$ such that $\a\in \{a_{12},a_{23},a_{34},a_{41},a_{15},a_{51}\}$. 
\begin{center}

\begin{tikzcd}[row sep=small, column sep=small]
                               &                                             &  & \lim\limits_\leftarrow \F \arrow[llddd, "t_1"] \arrow[dddd, "t_4"] \arrow[rddd, "t_2"'] \arrow[rrrdddd, "t_3"] \arrow[lllddd, "t_5"'] &                      &  &                       \\
                               &                                             &  &                                                                                                                                                &                      &  &                       \\
                               &                                             &  &                                                                                                                                                &                      &  &                       \\
K^\times \arrow[r, bend right, "s"'] & K^\times  \arrow[rrr,"s"] \arrow[l, "s"'] &  &                                                                                                                                                & K^\times \arrow[rrd,"s"'] &  &                       \\
                               &                                             &  & K^\times  \arrow[llu,"s"]                                                                                                                          &                      &  & K^\times  \arrow[lll,"s"]
\end{tikzcd}
\end{center}
Now, 
\begin{equation*}
    \begin{array}{R{1.3em}R{0.25em}l}
    \L{\F}  &= &\{(x_i)_1^5\colon x_j=\F(\a)(x_i),\forall \a\colon i\to j\text{ in } \I_E\}  \\
    &= &\{(x_i)_1^5\colon x_2=x_1^2, \ x_3=x_2^2, \ x_4=x_3^2, \ x_1=x_4^2, \ x_5=x_1^2, \ x_1=x_5^2\}.
    \end{array}
\end{equation*}
Eliminating parameters, we obtain that $x_1^3=1$ and consequently  
$$\L{\F}=\{(1,\dots,1),(\omega,\omega^2,\omega,\omega^2,\omega),(\omega^2,\omega,\omega^2,\omega,\omega^2)\}.$$
The isomorphism between $\L{\F}$ and $\mu_3(K)$ is clear.

Consider any evolution algebra $A$ which graph, relative to
a natural basis $\{u_i\}_{i=1}^4$, is \eqref{diaone}. So we have 
\begin{center}
\begin{tikzcd}[row sep=small, column sep=small]
u_5 \arrow[rd,controls={+(0,-0.5) and +(-0.9,0)} ]  &         &         \\
        & u_1\arrow[r] \arrow[lu,controls={+(.5,.5) and +(0.9,0.1)} ]&  u_2\arrow[d]\\
        & u_4  \arrow[u]&  u_3 \arrow[l],
\end{tikzcd}
\end{center}
and if we compute the diagonalizable automorphisms $f\colon A\to A$, that is, the automorphisms satisfying $f(u_i)=\l_i u_i$ ($i=1,\ldots, 5$, $\l_i\in K$), we find that
$f(u_1)=\omega u_1$, $f(u_2)=\omega^2 u_1$, $f(u_3)=\omega u_3$, 
$f(u_4)=\omega^2 u_4$, $f(u_5)=\omega^2 u_5$, where $\omega\in\mu_3(K)$. Thus,
the group of diagonalizable automorphisms of $A$ is (isomorphic to) $\mu_3(K)$ and therefore it is
$\lim\limits_\leftarrow\F$. If the evolution algebra with graph \eqref{diaone} is perfect, it is easy to check that any automorphism of $A$ is diagonalizable. So, in this case, we have 
$\aut(A)\cong\lim\limits_\leftarrow\F$.
\end{example} 

The idea issued by the previous paragraphs is the following: take an evolution algebra $A$ and let $\I_E$ be the small category associated to its graph relative to a natural basis $B$. 
Then the group of diagonalizable automorphisms of an evolution algebra $A$, relative to the basis $B$,
is isomorphic to the inverse limit of the functor $\lim\limits_\leftarrow\F$. Given that the functor $\F$ is ubiquitous, we give it a notorious
place in this work:

\begin{definition} \label{notorius}
     Taking into account
    the adjunction in formula \eqref{noitcnujda} and Remark~\ref{atnot},
     we define $\CC$ as the functor $\mathfrak{C}\colon \I_E\to \grp$ such that $\CC(i)=K^\times$ for any $i$ and $\CC(\a)\colon K^\times\to K^\times$ is the squaring map for any $\a \in E^1$.
\end{definition}

\begin{example} 
For a first example in the infinite-dimensional case, consider the category $\I_E$ induced by the graph in Figure \ref{fig_example_solenoide}, that is, the class of objects is $\N$ and 
$\hom(i+1,i)$ has cardinal one for any $i$. Let $\CC$ be the functor defined in Definition \ref{notorius}.  Then 
$\L{\CC}$  is the group of all sequences $(x_i)_{i\in\N^*}$ such that $x_{i+1}^2=x_i$ for any $i$. 

\[\begin{tikzcd}[row sep=small,column sep=small]
                      &  &                     &  & \lim\limits_\leftarrow\mathfrak{C} \arrow[rrrrdd, "t_1"] \arrow[rrdd, "t_2"] \arrow[dd, "t_i"] \arrow[lldd, "t_{i+1}"] &  &                     &  &          \\
                      &  &                     &  &                                                                                                                          &  &                     &  &          \\
{} \arrow[rr, dotted] &  & K^\times \arrow[rr,"s"'] &  & K^\times \arrow[rr,dotted]                                                                                                      &  & K^\times \arrow[rr,"s"'] &  & K^\times
\end{tikzcd}\]
Now, if $A$ is an evolution algebra whose graph relative to a basis $\{u_i\}_{i\in\N^*}$ is the one given in Figure $\ref{fig_example_solenoide}$, it is easy to see that the
diagonalizable automorphisms $f$ are of the form $f(u_i)=\l_i u_i$, where $\l_{i+1}^2=\l_i$.
So again the group of diagonalizable automorphisms of $A$ is isomorphic to $\L{\CC}=\{(x_n)_{n\ge 1}\colon x_{n+1}^2=x_n, n\ge 1\}$. 
If we take an element $(x_n)_{n\geq 1}$ in this group and define $a:=x_1$, then 
$x_2^2=x_1= a$ so $x_2\in\mu_{2}^{(a)}(K)$. In addition, $x_3^4=x_2^2=a$ and
$x_3\in\mu_{2^2}^{(a)}(K)$. In general
each $x_n$ is in $\mu_{2^{n-1}}^{(a)}(K)$. Consequently, $(x_n)_{n\geq 1}\in\mu_{2^\infty}^{(a)}(K)$. A moment's reflection reveals that $$\L{\CC}=\bigcup_{a\in K^\times}\mu_{2^\infty}^{(a)}(K).$$
\end{example}

\begin{example}\label{diomucho}   Let $A$ be a perfect evolution algebra with a natural basis
$B=\{u_i\}_{i \in \N^*}$ and multiplication $u_1^2=u_1$, $u_{i+1}^2=u_i$ for $i\ge 1$. This algebra is perfect and, as we will see, $\aut(A)$ is infinite (in contrast to the fact 
that perfect finite-dimensional evolution algebras have finite groups of automorphisms). Moreover, the structure matrix is neither invertible nor 2LI (see Section ~\ref{sec_2LI} for definitions). Even taking into account the results of Section~\ref{sec_non_diagonalizable}, we can not conclude that the group of automorphisms (relative to $B$) has the special form described in Corollary \ref{tesis}, nor can we say that it is diagonalizable. However, this is the case as the following proposition shows.

\begin{proposition}
    Any automorphism of $A$ is diagonalizable.
\end{proposition}

\begin{proof}
Let $\theta\in\aut(A)$. So, $\theta(u_i)=\sum_k\theta_{ik} u_k$ and the conditions for being a homomorphism imply that
$$\begin{cases}\theta_{i\a}\theta_{j\a}=0,\ \a\ge 3,\  i\ne j\\
\theta_{i1}\theta_{j1}+\theta_{i2}\theta_{j2}=0,\ i\ne j\\
\theta_{11}=(\theta_{11})^2+(\theta_{12})^2\\
\theta_{1k}=(\theta_{1(k+1)})^2,\ k\ge 2\\
\theta_{11}=(\theta_{21})^2+(\theta_{22})^2\\
\theta_{1k}=(\theta_{2(k+1)})^2,\ k\ge 2.
\end{cases}$$
Take $k\ge 2$ and think of $\theta_{1k}$. If 
$\theta_{1k}\ne 0$, the equations above imply that 
$\theta_{1(k+1)},\theta_{2(k+1)}\ne 0$. But $\theta_{1(k+1)}\theta_{2(k+1)}=0$, a contradiction.
Thus $\theta_{1k}=0$ if $k\ge 2$. The last equation
gives $\theta_{2k}=0$ for $k\ge 3$. Then the third equation gives $\theta_{11}=1$. The second gives $\theta_{j1}=0$ for $j\ne 1$. Now the second equation
gives $\theta_{21}=0$. So far $\theta(u_i)\in Ku_i$ for $i=1,2$. Next assume that $\theta(u_i)\in K u_i$ for 
$i=1,\ldots,q-1$. Then, for $q>2$ 
$$\theta(u_q^2)=\theta(u_{q-1})=\theta_{(q-1)(q-1)} u_{q-1}.$$
But $\theta(u_q)^2=\left (\sum_{k\geq 1}\theta_{qk} u_k\right )^2=\sum_{k\geq 1}(\theta_{qk})^2 u_k^2=
((\theta_{q1})^2 +(\theta_{q2})^2)u_1+\sum_{k>2}(\theta_{qk})^2 u_{k-1}$
which gives 
$$\begin{cases}\theta_{(q-1)(q-1)}=(\theta_{qq})^2,\\ 
(\theta_{q1})^2+(\theta_{q2})^2=0,\\
\theta_{qk}=0, (k\ne 1,2,q).
\end{cases}$$
Moreover, $0=\theta(u_1)\theta(u_q)=u_1\sum_{k\geq 1}\theta_{qk} u_k=\theta_{q1} u_1$ implying $\theta_{q1}=0$ (and similarly $\theta_{q2}=0$). In conclusion, $\theta(u_q)\in K u_q$.\end{proof}
\color{black}
Next, we investigate the group $\aut(A)$, which we know coincides with $\diag(A;B)$. The graph of the algebra is given by Figure \ref{fig_probando}.
\begin{figure}[h]
    \centering
    \begin{tikzcd}
	{} & \overset{u_3}{\bullet} & \overset{u_2}{\bullet} & \overset{u_1}{\bullet}
	\arrow[dashed, from=1-1, to=1-2]
	\arrow[from=1-2, to=1-3]
 \arrow[from= 1-3, to=1-4]
	\arrow[from=1-4, to=1-4, loop, in=55, out=125, distance=10mm]
\end{tikzcd}
\caption{}
    \label{fig_probando}
\end{figure}

Applying the functor $\CC$ we get the diagram given in Figure \ref{fig_probandoprobando}

\begin{figure}[h]
    \centering
    \begin{tikzcd}
	{} & K^\times & K^\times & K^\times
	\arrow[dashed, from=1-1, to=1-2]
	\arrow[from=1-2, to=1-3,"s"]
 \arrow[from= 1-3, to=1-4,"s"]
	\arrow[from=1-4, to=1-4, loop, in=55, out=125, distance=10mm,"s"]
\end{tikzcd}
\caption{}
    \label{fig_probandoprobando}
\end{figure}
\[\]
where again $s(x)=x^2$ for any $x\in K^\times$. Then, $\L{\CC}$ is the group of all
sequences $(x_i)_{i\in\N^*}$ such that $x_1=1$ and $x_{i+1}^2=x_i$ for $i\ge 1$.
{ Further, since $\aut(A)=\diag(A;B)\cong\L{\CC}\cong\L{\m_{2^n}(K)}$, we find that for this algebra, the automorphism group agrees with the so called Tate module of the group $K^\times$. The Tate module was introduced in \cite{Tate} in the context of abelian varieties. However, we will use the variant related to abelian groups. Following \cite{serre},
 the Tate module is associated with the $\ell$-adic completion of the torsion subgroup of an abelian group (for a prime $\ell$). The definition given in
 \cite{serre} is:
}
%

%

 %
 {
\begin{definition}
    Let $\A$ be an abelian group and $\ell$ a prime. The {\em Tate module} $T_\ell(\A)$ is the inverse limit:  
  \[
  T_\ell(\A) = \varprojlim_n \A[\ell^n],
  \]  
   where $\A[\ell^n] = \{ a \in \A \mid \ell^n a = 0 \}$ is the $\ell^n$-torsion subgroup, and the transition maps are given by multiplication by $\ell$.  
\end{definition}
In our case we can take $\A=K^\times$ so that the
$2^n$ torsion is $\m_{2^n}(K)$, that is, $\A[2^n]=\m_{2^n}(K)$ and consequently
$T_2(K^\times)=\L{\m_{2^n}(K)}=Aut(A)$. The group of
automorphisms is the Tate module $T_2(K^\times)$,
so we can take advantage of the known values of this. Below we describe $T_2(K^\times)$ for various fields $K$.
 
\medskip

\begin{enumerate}
    \item Algebraically closed fields of $\car(K)\ne 2$:
    \begin{equation*}
        T_2(K^\times) \cong \mathbb{Z}_2,
    \end{equation*}
   the free rank-1 $\mathbb{Z}_2$-module. See \cite[Section 5.3]{dembele2023honda}.
   
    \item Finite fields ($ K = \mathbb{F}_{p^n}$ with $p$ odd and $n\ge 1$):\\
   \begin{equation*}
       T_2(\mathbb{F}_{p^n}^\times)= 1.
   \end{equation*} 
   
    \item Fields with/without $2^n$-roots of unity:
\begin{itemize}
    \item If $K$ contains no nontrivial $2^n$-th roots of unity for any $ n \geq 1 $ (e.g., $ \mathbb{Q} $ or $ \mathbb{Q}_p $ with $ p \neq 2 $): 
    \begin{equation*}
        T_2(K^\times) = 1.
    \end{equation*}
    
    \item If $ K $ contains $ \sqrt{-1} $ but no higher $ 2^n$-roots (e.g. $ K = \mathbb{Q}(i) $):
    \begin{equation*}
        T_2(K^\times) =1.
    \end{equation*}
   \item If $K$ contains all $2^n$-roots of unity (e.g. $K=\mathbb{Q}(\zeta_{2^\infty}):=\bigcup_{n\ge 1}\mathbb{Q}(\zeta_{2^n})$ with $\zeta_n$ primitive $n$-th root of $1$): 
   \begin{equation*}
       T_2(K^\times)\cong\Z_2.
   \end{equation*}
\end{itemize}
    \item $\car(K) = 2$:
    \begin{equation*}
        T_2(K^\times) = 1
    \end{equation*}
    since $x^{2^n} = 1$ implies
   $x = 1$.
\end{enumerate}

Not every value of $T_2(K^\times)$ above is straightforward to justify. Therefore, we now provide some hints to support and verify these.

The case of characteristic $2$ is obvious. More generally if $K$ has characteristic other than $2$ and $\cup_{n\ge 1}\m_{2^n}(K^\times)$ is a finite set, we have $T_2(K^\times)=1$. Indeed, The chain
$$\m_2(K)\subseteq\m_{2^2}(K)\subseteq\cdots\subseteq\m_{2^i}(K)\subset\m_{2^{i+1}}(K)\subseteq\cdots$$
is stationary hence we may assume that $\m_{2^i}(K)=\m_{2^{j}}(K)$ for $j\ge i$, where $i$ is the
stationary index of the chain. Now, take any element $(x_i)\in\L{\m_{2^n}(K^\times)}$. We know
that $x_{j+1}^2=x_j$ for any $j$. Thus, $x_{j+k}^{2^k}=x_j$ for any $j$ and $k$. Take $k>i$ (the stationary index of the above chain). Then $x_{j+k}\in\m_{2^{j+k}}(K)=\m_{2^i}(K)$, because $j+k>i$. Thus $x_{j+k}^{2^i}=1$ and
$x_j=x_{j+k}^{2^k}=(x_{j+k}^{2^i})^{2^{k-i}}=1$. Consequently, $T_2(K^\times)=1$.

If $K/\Q$ is an extension containing all the roots
$\zeta_{2^n}=\exp(2\pi i/2n)$ for $n\ge 1$, then 
$\m_{2^n}(K)=\{1,\zeta_{2^n},\zeta_{2^n}^2,\ldots,\zeta_{2^n}^{2^n-1}\}\cong C_{2^n}$ and the inverse limit is the group of $2$-adic integers.
}
\end{example}

 We conclude this section by proving that the group of diagonalizable automorphisms of an evolution algebra is the inverse limit of the functor $\CC\colon \I_E\to\grp$ (recall Definition \ref{notorius}).

\begin{theorem}\label{arop}
    Let $A$ be any evolution $K$-algebra, fix a natural basis $B$, and consider the graph $E$ associated with $B$. Then, there is a group isomorphism $\displaystyle\diag(A;B)\cong\lim\limits_{\leftarrow}\CC$.
\end{theorem}
\begin{proof}
For any $ u \in E^0 $, define $\pi_u \colon \diag(A; B) \to \CC(u) = K^\times$ by assigning to each $\theta \in \diag(A;B)$ the scalar such that $\theta(u) = \pi_u(\theta)u$. It is easy to check that $\pi_u$ is a group morphism. Consider an edge $f$ with source $u$ and range $v$. We check the commutativity of the triangles
\begin{center}
\begin{tikzcd}[cramped, column sep=small]
\diag(A;B) \arrow{r}{\pi_u}  \arrow[rd,"\pi_v"'] 
  & \CC(u) \arrow{d}{\CC(f)} & x\arrow[d]\\
    & \CC(v) & x^2.
\end{tikzcd}
\end{center}
Take $\theta\in\diag(A;B)$ with $\theta(u)=\pi_u(\theta)u$, while $\theta(v)=\pi_v(\theta)v$. Since the edge $f$ connects $u$ to $v$ we know that $u^2= m v+\xi$, where $m\in K^\times$ and $\xi v=0$. But, 
$$\pi_u(\theta)^2(mv+\xi)=\theta(u)^2=\theta(u^2)=\theta(mv+\xi)=m\theta(v)+\theta(\xi)=m\pi_v(\theta)v+\theta(\xi),$$ so $\pi_u(\theta)^2 m v=m \pi_v(\theta) v$ because $\theta(\xi)v=\pi_v(\theta)^{-1}\theta(\xi v)=0$. Consequently, $\pi_u(\theta)^2=\pi_v(\theta)$, that is, $\CC(f)(\pi_u)(\theta) = \pi_v(\theta)$ and the commutativity of the triangles above is proved. So, $\diag(A;B),\{\pi_u\}_{u\in E^0}$ is a cone for the functor $\CC$. 

To end the proof, we must check that the cone is terminal. For this, assume that $G$ is another group endowed with homomorphisms $t_u\colon G\to\CC(u)$ (for any $u\in E^0$) such that $\CC(f)t_u=t_v$ whenever $f$ is
an edge with $s(f) = u$ and $r(f) = v$. We have to prove that there is a unique group homomorphism $t\colon G\to\diag(A; B)$ such that $\pi_u t=t_u$ for any vertex $u$. 
Notice that for any $ g \in G $, we have a nonzero scalar $ t_u(g) \in K^\times $, and we define an automorphism $ \theta \colon A \to A $ by setting $ \theta(u) = t_u(g)u $ for all $ u \in E^0 $. Then, we define $ t(g) = \theta $ and, since $ \pi_u(t(g)) = \pi_u(\theta) = t_u(g) $, it follows that $ \pi_u \circ t = t_u $.
The fact that $t$ is a group homomorphism, as well as its uniqueness property, is easy to check. 
\end{proof}

\section{Deviations from the finite-dimensional case}\label{sec_deviation_finitedimensional}

Motivated by the distinct behaviours of the automorphism groups in finite-dimensional and infinite-dimensional evolution algebras, in this section, we highlight some differences in the behaviour of infinite-dimensional evolution $ K $-algebras compared to the finite-dimensional case.

\subsection{Perfection {\rm x} Nondegeneracy} While in the finite-dimensional case, the perfection of evolution algebras implies nondegeneracy, in the infinite-dimensional case, there are perfect evolution algebras that are degenerate. For instance,
the $K$-algebra $A=\Span{(\{e_i\}_{i\geq 0})}$ with a countable basis whose multiplication is $e_0^2=0$, $e_{i+1}^2=e_i$ and $e_ie_j=0$ when $i\ne j$. This algebra is degenerate, but $A^2=A$. Thus, the only hope to extend some finite-dimensional results regarding perfect algebras to arbitrary dimensions will likely require replacing the perfection condition with a stronger assumption.

\subsection{Perfection {\rm x} Finiteness of the automorphism group.}\label{soperfect} 
As it is well known, the automorphism group of finite-dimensional perfect
evolution algebras is finite (see \cite{Elduque}). In the infinite-dimensional case, this is no longer true. For instance, in the algebra $A$ of the previous item (which is perfect), for any diagonalizable automorphism $f\colon A\to A$ determined by a sequence $(x_i)_{i\geq 0}$ in $K$, that is, $f$ given by $f(e_i)=x_i e_i$ for every $i\geq 0$, 
we have  $x_ie_i=f(e_i)=f(e_{i+1}^2)=f(e_{i+1})^2=x_{i+1}^2 e_i$, so that $x_i=x_{i+1}^2$ for any $i\geq 0$. Thus,
the group $\aut_K(A)$ contains the subgroup of all the sequences $(x_i)_{i\ge 0}$ satisfying  $x_i=x_{i+1}^2$ for any $i \geq 0$. This subgroup is nothing but the inverse limit of the system

$$\cdots\to K^\times\to K^\times \to K^\times\to K^\times,$$
where all arrows are the squaring map $x\mapsto x^2$. So, $\displaystyle\lim_{\leftarrow}K^\times\subset \aut\nolimits_K(A)$ as a subgroup. If we take the complex numbers as the ground field, say $K=\complex$, then
$\displaystyle\lim_{\leftarrow}\complex^\times$ contains the circle $S^1$, since we can consider the group monomorphism $S^1\to \displaystyle\lim_{\leftarrow}\mathbb{C}^\times$ given by $\exp(\iu t)\mapsto (\exp(\iu t/2^j))_{j\ge 0}$, for any $j \geq 0 $.

Furthermore, we have the following commutative diagram.

\[\begin{tikzcd}
	{{\displaystyle \lim_{\leftarrow} \mathbb{C}^\times}} & {{\mathbb{C}^\times}} & {\mathbb{C}^\times} & {{\mathbb{C}^\times}} & {{\mathbb{C}^\times}} \\
	{\displaystyle\lim_{\leftarrow} S^1} & {S^1} & {S^1} & {{S^1}} & {{S^1}}
	\arrow[from=1-1, to=1-2]
	\arrow[dotted, from=1-2, to=1-3]
\arrow[hook, from=2-2, to=1-2]
\arrow[hook, from=2-3, to=1-3]
	\arrow["s", from=1-3, to=1-4]
	\arrow["s", from=1-4, to=1-5]
	\arrow[hook, from=2-1, to=1-1]
	\arrow[from=2-1, to=2-2]
	\arrow[dotted, from=2-2, to=2-3]
	\arrow["s"', from=2-3, to=2-4]
	\arrow[hook, from=2-4, to=1-4]
	\arrow["s"', from=2-4, to=2-5]
	\arrow[hook, from=2-5, to=1-5]
\end{tikzcd}\]
As it is well known, the inverse limit
$\displaystyle\lim_{\leftarrow} S^1$ is the dyadic solenoid, see Figure~\ref{fig_solenoide}. Since $\displaystyle \lim_{\leftarrow}S^1\subset\lim_{\leftarrow}\mathbb{C}^\times\subset \aut(A)$, we conclude that $\aut(A)$ contains the dyadic solenoid $\displaystyle\lim_{\leftarrow}S^1$.

We have described how the diagonalizable automorphisms of an evolution algebra $A$ can contribute to the infiniteness of the group $\aut(A)$. But non-diagonalizable automorphisms can also do it.
For instance, consider the evolution $K$-algebra $K^{(\N)}$ of
sequences $(x_n)_{n\ge 0}$ in $K$ with finitely-many nonzero entries. We consider here the component-wise product. Denote the canonical basis by $\{e_i\}_{i\ge 0}$. Then, for any bijection
$\s\colon\N\to\N$, we have an automorphism of $K^{(\N)}$ given by
mapping $e_j$ to $e_{\s(j)}$ for any $j$. In this case, $\aut(K^{(\N)})$ contains a copy of the group of bijections $\N\to\N$ (an infinite "permutation group").
\subsection{Perfection {\rm x} Inversibility of the structure matrix}\label{subsec_inversibility} If $B=\{e_i\}_{i\in\Lambda}$ is a natural basis of an evolution $K$-algebra and $M_B$ is the associated structure matrix, then $M_B$ is column finite (a matrix whose columns have finitely many nonzero entries).
Recall that column finite matrices can be multiplied and inversibility makes sense. 
When $M_B$ is invertible, we have $A=A^2$. However, the same example above shows that the perfection of evolution algebras does not imply the inversibility of the structure matrix of the algebra. In any case, we have:

\begin{proposition}\label{qnu}{\rm (c.f. \cite[Theorem 4.4]{Elduque2})} Let $A$ be any evolution algebra (possibly infinite-dimensional) and
assume that for a natural basis $B=\{e_i\}_{i\in\Lambda}$ of $A$, the
structure constants matrix $M_B$  is invertible. Then, for any natural basis
$\{u_i\}_{i\in \Lambda}$  there is a bijection $\sigma\colon \Lambda\to \Lambda$ and scalars $\{k_i\colon i\in \Lambda\}\subset K^\times$ such that $u_i=k_i e_{\sigma(i)}$ for any $i\in \Lambda$. 
\end{proposition}
\begin{proof}
The proof is essentially an adaptation of \cite[Lemma 4.3 and Theorem 4.4]{Elduque2}. Write $u_i=\sum_k a_{ki} e_k$ so that for $i\ne j$ we have
$0=u_iu_j=\sum_ka_{ki} e_k\sum_h a_{hj} e_h=\sum_{k}a_{ki} a_{kj} e_k^2$. Observe that $\{e_k^2\}_{k\in \Lambda}$ is a linearly independent set. Indeed, if there is a linear combination $\sum_{i \in \Lambda_0} \lambda_i e_i^2 = 0$ with $\lambda_i \in K$, then there is a column vector $\lambda = (l_i)^T_{i\in \Lambda}$ ($T$ denote the transpose of the matrix) with $l_i=\lambda_i $ if $i \in \Lambda_0$ and $l_i=0$ in other case. In this way,  $M_B\lambda = 0$ which is only possible if $\lambda_i = 0$ for all $i \in \Lambda_0$.

So, we have $a_{ki}a_{kj}=0$ whenever $i\ne j$ for any $k$. Therefore, if some $a_{ki}\ne 0$, it follows that $a_{kj}=0$ for any $j\ne i$. Given that the matrix of change of basis $(a_{ij})$ is invertible, we conclude that each row and column must have exactly one nonzero entry.
Consequently, there exists a permutation $\sigma$ of the index set $\Lambda$ such that $u_i=a_{ \sigma(i)i}e_{\sigma(i)}$ for every $i$. 
\end{proof}

\begin{corollary}\label{dostroc}
    If $ A $ is an evolution algebra with invertible structure matrix $ M_B $, then any automorphism $ f \colon A \to A $ is of the form $ f(e_i) = k_i e_{\sigma(i)} $ for each $ i $, where $ B = \{e_i\}_{i \in \Lambda} $ and $ \sigma $ is a permutation of $ \Lambda $.
\end{corollary}

\section{The 2LI condition}\label{sec_2LI}

An evolution algebra satisfying 2LI is defined in \cite[Definition 2.8]{natural}. Accordingly, we say that a natural basis \( B \) satisfies 2LI if, for any two distinct elements \( x, y \in B \), the set \( \{x^2, y^2\} \) is linearly independent. In this section, we will show that if there is a natural basis satisfying 2LI, every other natural basis coincides with the first one up to permutation and scalar multiplication. Thus, the group \( \operatorname{Diag}(A) \) does not depend on the choice of basis; see Remark~\ref{vespa}.

Recall that an element of an evolution algebra is called \textit{natural} if it belongs to some natural basis. The characterization of naturality for vectors is provided in \cite{natural}.

\begin{proposition}
    Let $A$ be an evolution $K$-algebra. Suppose that \( \{e_i\}_{i \in \Lambda} \) is a natural basis of \( A \) and let \( u = \lambda_1 e_1 + \cdots + \lambda_k e_k \), where \( \lambda_i \neq 0 \).
    \begin{enumerate}[\rm (i)]
    \item If  $u^2\ne 0$, then:
    \begin{enumerate} 
        \item If $\car(K) \neq 2$ and $\dim(\Span(\{e_i^2\}_{i=1}^k))=1$, then $u$ is a natural vector.
        \item If  $u$ is natural, then $\dim(\Span(\{e_i^2\}_{i=1}^k))=1$.
    \end{enumerate}
    \item If $u^2= 0$ then $u$ is natural if and only if $e_i^2=0$ for $i=1,\ldots,k$.
\end{enumerate}
\end{proposition}

   \begin{proof}
       Assume that $u^2\neq 0$, $\car{K} \neq 2$ and  $\dim(\Span(\{e_i^2\}_{i=1}^k))=1$. Let \( S = \Span(\{e_i\}_{i=1}^k) \). We suppose, without loss of generality, that $e_1^2$ generates $\Span(\{e_i^2\}_{i=1}^k)$. There exists a  symmetric bilinear form \( \langle \cdot, \cdot \rangle \) defined on \( S \) such that 
\(
xy = \langle x, y \rangle e_1^2.
\)
Since \( \langle u, u \rangle \neq 0 \), we have \( S = Ku \perp (Ku)^\perp \). Furthermore, as the characteristic of $K$ is not $2$, \( (Ku)^\perp \) has an orthogonal basis \( \{v_j\}_{j=1}^{k-1} \). Consequently, there exists a natural basis \( \{u\} \cup \{v_j\}_{j=1}^{k-1} \cup \{e_i\}_{i > k} \) that contains \( u \). 

For item (1) (b), if \( u \) is a natural vector, then \( \dim(\operatorname{Im}(L_u)) = 1 \) where $L_u$ is the left-multiplication operator. On the other hand, \( \operatorname{Im}(L_u) \) is generated by \( u e_i \) for \( i \in \Lambda \), that is, 
\(
\operatorname{Im}(L_u) = \Span(\{e_1^2, \ldots, e_k^2\}).
\)
Hence, \( \dim (\Span(\{e_i^2\}_{i=1}^k)) = 1 \).

To show item (2), start with $u$ being natural. Then \( uA = 0 \), and thus \( ue_i = 0 \), which implies \( e_i^2 = 0 \) for \( i = 1, \dots, k \). Conversely, if each $e_i^2=0$ then $\{u\}\cup\{e_i\}_{i=1}^{k-1}\cup\{e_i\}_{i>k}$ is a natural basis containing $u$.
   \end{proof}

\begin{remark} 
The hypothesis that $\car(K) \neq 2$ in the first item above can not be dropped. Consider the evolution algebra $K^3$ with $K$ a field of characteristic $2$. In addition, if we suppose that the canonical basis $\{e_1,e_2,e_3\}$ of $K^3$ has products $e_1^2=e_2^2=e_3^2\ne 0$ and $e_ie_j=0$ when $i\ne j$, then we obtain a family of evolution algebras. In any of these algebras, take $u=e_1+e_2+e_3=(1,1,1)$ whose square is $u^2=e_1^2+e_2^2+e_3^2=3e_1^2=e_1^2\ne 0$. Also, $\hbox{supp}(u)=\{1,2,3\}$ and $\dim(\Span\{e_1^2,e_2^2,e_3^2\})= 1$. However, $u$ is not natural. If $\{(1,1,1), (x,y,z), (r,s,t)\}$ is a natural basis, we must have
$$\begin{cases}x+y+z=0\cr r+s+t=0\cr xr+ys+zt=0\end{cases},\quad \begin{vmatrix}1 & 1 & 1\cr x & y & z\cr r & s & t\end{vmatrix}\ne 0$$ 
which is not consistent in $\car K =2$. Indeed, as $z= -x-y$ and $t=-r-s$, then replacing in the third equation and the determinant, we have $2 r x + s x + r y + 2 s y = sx+ry = 0$ and $3 s x - 3 r y =sx+ry\neq 0$,  a contradiction.
\end{remark}
Thus, in item (i) of \cite[Theorem 2.4]{natural}, the hypothesis that $\hbox{char}(K)\ne 2$ must be added.


As noted in Proposition \ref{qnu}, if the structure matrix of \( A \) is invertible, then any two natural bases coincide up to permutation and scalar multiplication. Next, we generalize \cite[Corollary 2.7]{natural}.

\begin{theorem} \label{2LI}
    Let $A$ be an evolution algebra (arbitrary dimension and ground field). If $A$ has  a natural basis satisfying the condition 2LI, then all the bases coincide up to permutation and scalar multiplication.
\end{theorem}
\begin{proof}
Assume that $B=\{e_i\}_{i\in\Lambda}$ is a natural basis satisfying 2LI and $B'$ is any other
natural basis. Let $u\in B'$ and write, without loss of generality, $u=\l_1e_1+\cdots+\l_ke_k$ with $\l_i\ne 0$.
Since $u$ is natural $\dim(\hbox{Im}(L_u))\le 1$, but $\hbox{Im}(L_u)=\Span(\{e_i^2\}_{i=1}^k)$. Consequently, if $k>1$, the dimension of $\hbox{Im}(L_u)\ge 2$, which is a contradiction. This implies that $k=1$ and hence $u\in K^\times B$. Thus, any element of $B'$ is a nonzero multiple of some element in $B$. 

\end{proof}

As noted in \cite{natural}, if \( A \) has a unique basis (up to reordering and scalar multiplication), it does not automatically follow that it satisfies 2LI. The first additional condition required for \( A \) is nondegeneracy. Moreover, we must assume that the ground field has more than three elements. Under these assumptions, the proof of \cite[Corollary 2.7]{natural} gives that:

\begin{proposition}Let $A$ be a nondegenerate evolution algebra of arbitrary dimension over a field with more than $3$ elements such that all the natural bases coincide up to permutation and scalar multiplication. Then all the natural bases satisfy 2LI.
\end{proposition}
Note that the hypothesis on the cardinality of the ground field is not present in the statement of \cite[Corollary 2.7]{natural}; however, it is used in its proof.

\section{Non-diagonalizable automorphisms}\label{sec_non_diagonalizable}
Let $\mathcal{B}$ be the set of all natural bases of a fixed evolution $K$-algebra $A$. We assume all natural bases are indexed by the same set $\Lambda$. Then we define the direct product group $S_\Lambda \times (K^\times)^\Lambda$ with pointwise multiplication, where $S_\Lambda$ is the group of bijections $\Lambda \to \Lambda$ (under composition) and $(K^\times)^\Lambda$ is the group of maps $\Lambda \to K^\times$ with elements represented in the form $(x_i)_{i \in \Lambda}$, where $x_i \in K^\times$ (with pointwise multiplication).
This group has a natural action $(S_\Lambda \times (K^\times)^\Lambda) \times \mathcal{B} \to \mathcal{B}$ defined as follows: if $\{e_i\}_{i \in \Lambda}$ is a natural basis of $A$ and $(\sigma, (x_i)_{i \in \Lambda}) \in S_\Lambda \times (K^\times)^\Lambda$, then
$$(\s,(x_i)_{i\in\Lambda})\cdot\{e_i\}_{i\in \Lambda}:=\{x_i e_{\s(i)}\}_{i \in \Lambda}$$
\begin{definition}  
The set of orbits on $\mathcal{B}$ induced by the action of the group $S_\Lambda \times (K^\times)^\Lambda$ will be denoted by $\overline{\mathcal{B}}$. For a given natural basis $B$, its orbit under this group will be denoted by $[B]$.
\end{definition}

\begin{remark}\label{vespa}
Under the previous settings, we have:

\begin{enumerate}[\rm (i)]
    \item If two basis $B,B'\in\B$ of an evolution algebra $A$ are in the same orbit under the action of $S_\Lambda\times (K^\times)^\Lambda$, then \begin{equation}\label{Bpedni}\diag(A;B)=\diag(A;B').\end{equation} 
    \item If $A$ satisfies 2LI, then the action of  $S_\Lambda\times (K^\times)^\Lambda$ is transitive. Consequently, there is only one orbit.
\end{enumerate}
\end{remark}
If the set of orbits $\overline\B$ has cardinal $1$, then we can speak of the group $\diag(A)$ without further allusion to any particular $B\in\overline\B$.

\begin{example}  
    Consider the $2$-dimensional evolution algebra with natural basis $\{e_1,e_2\}$ such that $e_1^2=e_1$ and $e_2^2=0$. Its associated graph corresponds to $E$.
    \vspace{-0.1cm}
               \begin{figure}[h]
        \centering
        \begin{tikzpicture}[shorten <=3pt,shorten >=3pt,>=latex,node distance={20mm}, main/.style = {draw, fill, circle, inner sep = 1pt}, sub/.style = {draw = Red, fill = Red, circle, inner sep = 1pt}]
\node (3) at (-1.5,0.5) {$E:$};
\node[main,label=above:$e_1$] (1) at (0,0) {}; 
\node[main,label=above:$e_2$] (2) at (1,0) {};

\draw[->] (1) to [out = 140, in = 210, looseness = 40] (1);

\end{tikzpicture}
    \end{figure}
 \vspace{-0.5cm}
    
    Let us search for all the possible natural bases (modulo the action of $S_2\times(K^\times)^2$).
    If $\{u_1,u_2\}$ is another natural basis, then we can write $u_1=xe_1+ye_2$ and $u_2=ze_1+t e_2$ ($x,y,z,t\in K$). Thus, $0=u_1u_2$ gives $xz=0$. Modulo the action of $S_2\times(K^\times)^2$, we may assume $z=0$. Since $u_1$ and $u_2$ are linearly independent, $xt\ne 0$, and hence the orbits of possible natural bases are of the form 
    
    $$\begin{cases}u_1= e_1+y e_2\\ u_2=e_2.\end{cases}$$
Now, if $y=0$ then the basis $\{u_1,u_2\}$ is in the orbit of $\{e_1,e_2\}$, but if $y\ne 0$ then they are in different orbits.
It is also easy to realize that different values of $y\in K$ produce different orbits. So,
\begin{equation*}
    \overline\B= \left\{\left [\{e_1+y e_2,e_2\}\right ] \colon y \in K\right\}.
\end{equation*}
\end{example}

\begin{proposition}
Given  an evolution algebra $A$, if  $[B], [B'] \in \overline\B$ with $[B]= [B']$, then the graphs associated to $A$ relative to $B$ and $B'$ are isomorphic. 
\end{proposition}

\begin{proof}
 Let $B=\{e_i\}$ and $B'=\{u_i\}$ be two natural bases with  $[B]= [B']$. We can write $u_i^2=\omega_{ji}u_j + C$ with $\omega_{ji}\neq 0$ and $C u_j=0$. On the other hand, we have that $u_i=\lambda_ie_{\sigma(i)}$ for some $\lambda_i \in K^{\times}$ and $\sigma \in S_\Lambda$. Therefore, $\lambda_i^2e_{\sigma(i)}^2=\omega_{ji}u_j + C$. This implies that $e_{\sigma(i)}^2=\frac{\lambda_j \omega_{ji}}{\lambda_i^2}e_{\sigma(j)} + \frac{C}{\lambda_i^2}$ with $\frac{C}{\lambda_i^2}e_{\sigma(j)}=0$. Hence, the graphs associated to A relative to $B$ and $B'$ are isomorphic.

\end{proof}

\begin{definition}  
    We define $\E_K$ as the category of pairs $(A, B)$ where $A$ is an evolution $K$-algebra and $B$ is a natural basis of $A$. A homomorphism from $(A,B)$ to $ (A',B')$ in $\E_K$ is a $K$-algebra isomorphism $t\colon A\to A'$ such that for any $b\in B$, there is some $b'\in B'$ with $t(b)\in Kb'$.  We will use
the notation $\aut_{\E_K}(A,B)$ for the group of all automorphisms of the object $( A, B)$ of $\E_K$. If there is no ambiguity, we will shorten the notation $\aut_{\E_K}(A, B)$ to $\aut(A, B)$.
\end{definition}

So fa,r we have considered automorphisms that are diagonalizable relative to a natural basis of the given evolution algebra. But in some cases the "symmetry" of the graph associated to the evolution algebra permits "twisted" automorphisms. To develop further related results, we need the notion of a $K$-{\it weighted graph}. This is a pair $(E,w)$, where $E$ is a graph and $w\colon E^1\to K^\times$. Typically, in applications, the field $K$ is taken to be the real numbers, but in our case, we consider an arbitrary field $K$.

\begin{remark}\label{laud}   
There is a biunivocal correspondence between pairs $(A, B)$, where $A$ is an evolution $K$-algebra and $B$ is a natural basis, and weighted graphs $(E, w)$ with $E$ satisfying Condition (Sing). 
Given an evolution algebra $A$ with a natural basis $B = \{u_i\}_{i \in \Lambda}$, we construct the graph $E$ such that $E^0 = B$, and there is exactly one edge from $u_i$  to $u_j$ if and only if 
with $\omega_{ji} \neq 0$. We define the weight $w\colon E^1 \to K^\times$ by setting $w(f) = \omega_{ji}$ if $f \in  s^{-1}(u_i) \cap r^{-1}(u_j)$.
Conversely, given a weighted $K$-graph $(E, w)$ satisfying Condition~(Sing), we can construct an evolution $K$-algebra with natural basis $E^0$ and multiplication defined by 
\[
e_i^2 = \sum_{f \in s^{-1}(e_i)} w(f) r(f)
\]
for each $i \in \Lambda$.
\end{remark}

We find ourselves searching for a  description of the category of weighted graphs in such a way that it is isomorphic to the category $\E_K$. Considering $(A,B)$ and $(A',B')$ with $B = \{e_i\}_{i \in \Lambda}$ and $B' = \{e_i'\}_{j\in \Lambda'}$, notice that if $\theta \colon (A,B) \to (A',B')$ is a morphism in $\E_K$, then
\begin{equation*}
    \theta(e_i^2) = \theta\left (\sum_{k} \omega_{ki}e_k\right ) = \sum_{k} \omega_{ki} \theta(e_k) = \sum_{k} \omega_{ki} x_ke_{\sigma(k)}',
\end{equation*}
\begin{equation*}
    \theta(e_i)^2 = x_i^2e_{\sigma(i)}'^2 = x_i^2\sum_{k}\omega'_{k\sigma(i)} e_k' 
\end{equation*}
and
\begin{equation*}
    x_i^2\omega_{\sigma(k)\sigma(i)}' = x_k\omega_{ki}
\end{equation*}
for every $k \in \Lambda$. Using the notation given in Remark~\ref{laud}, this motivates the construction of the ideal $I_\theta$ in the next definition.

\begin{definition}  
We define the category $\graph_K$ whose objects are the weighted graphs $(E,w)$. To define homomorphisms, consider
two weighted graphs $(E,w)$ and $(E',w')$. Denote the elements of $E^0$ by  $\{e_i\}_{i\in\Lambda}$.
Take the polynomial algebra $K[x_i,y_i\colon i\in\Lambda]$ and, for any homomorphism $\theta\colon E\to E'$ in $\graph$, define  the ideal $I_\theta\triangleleft K[x_i,y_i\colon i\in\Lambda]$ as the one generated by all polynomials in the set:
\begin{equation}\label{ahi}
S:=\{x_iy_i-1\}_{i\in\Lambda}\cup\{w(a)x_j-w'(\theta(a))x_i^2\colon a\in E^1\cap s^{-1}(e_i)\cap r^{-1}(e_j)\}.
\end{equation}
So,  the zero-loci of the ideal $I_\theta$ is the set $V(I_\theta)$ given by
$$\left\{((x_i),(y_i))\in (K^\Lambda)^2\colon
x_iy_i=1, w(a)x_j=w'(\theta(a))x_i^2, \ \forall a\in s^{-1}(e_i)\cap r^{-1}(e_j) \right\}.$$
A homomorphism $(E,w)\to (E',w')$ in $\graph_K$  is defined to be a morphism $\theta_\vec$, where $\theta_\vec\colon E\to E'$ is a graph isomorphism and $\vec=(x_i)_{i\in\Lambda}\in (K^\times)^\Lambda$ is such that $((x_i),(x_i^{-1}))\in V(I_{\theta_\vec})$.
\end{definition}

Note that if $\theta\colon E\to E'$ is such that $w'(\theta(a))=w(a)$ for 
any arrow in $E^1$, then define $\vec:=(x_i)$ where $x_i=1$ (for all $i$). Hence $(\vec,\vec)$ is in $V(I_{\theta})$ and $\theta=\theta_\vec$.

We will show that the categories $\E_K$ and $\graph_K$ are isomorphic, but we need two auxiliary results for this.

\begin{lemma}\label{previous}
  Let $(A,B), (A',B')\in\E_K$ with $B=\{u_i\}_{i\in\Lambda}$ and $B'=\{v_i\}_{i\in\Lambda'}$ and consider $t\colon (A,B)\to (A',B')$ a homomorphism in $\E_K$ such that
  $t(u_i)=x_iv_{\s(i)}$ ($i\in\Lambda$), where $x_i\in K^\times$ and $\s\colon\Lambda\to\Lambda'$. If $(E,w)$ 
  and $(E',w')$ are the weighted graphs associated to $(A,B)$ and $(A',B')$ respectively, then there is a homomorphism $\theta_\vec\colon (E,w)\to (E',w')$ such that $\theta_\vec(u_i)=v_{\s(i)}$.
\end{lemma}
\begin{proof}
Let $t$ be the same as the statement of the lemma. Define $\theta\colon E^0\to E'^0$ by $\theta(u_i)=v_{\s(i)}$ for any $i$. If $a\in E^1\cap s^{-1}(u_i)\cap r^{-1}(u_j)$,
we have $u_i^2=w(a)u_j+R$ with $w(a)\ne 0$ and $u_jR=0$. Thus, $t(u_i)^2=w(a)t(u_j)+t(R)$, that is, $x_i^2 v_{\s(i)}^2=w(a)x_jv_{\s(j)}+t(R)$.  Equivalently,
\begin{equation}\label{wellington}
v_{\s(i)}^2=w(a)\frac{x_j}{x_i^2}v_{\s(j)}+\frac{t(R)}{x_i^2},
\end{equation} where $v_{\sigma(j)}t(R) = 0$ because $u_jR=0$. So, there is an  edge in $E'^1$
connecting $v_{\s(i)}$ to $v_{\s(j)}$. We formally define $\theta\colon E^1\to E'^1$ by declaring $\theta(a)$ as the  edge connecting $v_{\s(i)}$ to $v_{\s(j)}$. Note that,
by construction, $s(\theta(a))=\theta(s(a))$ and $r(\theta(a))=\theta(r(a))$.  Thus, $\theta$ is an
isomorphism between $E$ and $E'$ in $\graph$.
On the other hand, by comparing \eqref{wellington} with 
\( v_{\s(i)}^2 = w'(\theta(a))v_{\s(j)} + R' \) 
(where we also have \( v_{\s(j)}R' = 0 \)), we obtain 
\( w(a)\frac{x_j}{x_i^2} = w'(\theta(a)) \), which implies $((x_i),(x_i^{-1}))
\in V(I_\theta)$.
 So, $\theta_\vec:= \theta$  is a homomorphism in $\graph_K$ with $\vec = (x_i)$.
\end{proof}

By Lemma \ref{previous}, we can define a functor \(\mathcal{F} \colon \E_K \to \graph_K\) as follows: for each object \((A, B)\) in \(\E_K\), we set \(\mathcal{F}(A, B) = (E, w)\), where $E$ is the graph
associated to $A$ relative to $B$ (as a set $E^0=B$). For any edge $a$ from $u$ to $v$ in this graph, we define $w(a)$ to be the scalar such that $u^2=w(a)v+R$, (with $Rv=0$). Furthermore, 
for each morphism \(t\), we define \(\mathcal{F}(t) = \theta_\vec\), where \(\theta_\vec\) is the homomorphism established in Lemma \ref{previous}.

\begin{lemma}\label{erehiuqa}
Let $\theta_\vec \colon (E,w)\to (E',w')$ be a homomorphism in $\graph_K$. Then there is a unique
homomorphism $t\colon (A,B)\to (A',B')$ in $\E_K$ such that $\mathcal{F}(t)=\theta_\vec$. 
\end{lemma}
\begin{proof}

Let \(E^0 = \{u_i\}_{i \in \Lambda}\) and \(E'^0 = \{v_i\}_{i \in \Lambda'}\). Then \(\theta_\vec(u_i) = v_{\sigma(i)}\) for some bijection \(\sigma \colon \Lambda \to \Lambda'\). Let $\vec = (x_i)$, it follows that  
\(
w(f)x_j = w'(\theta_{\vec}(f))x_i^2
\)  
for all \(f \in E^1 \cap s^{-1}(u_i) \cap r^{-1}(u_j)\). This enables us to define \(t \colon A \to A'\) as the linear extension of \(t(u_i) = x_i v_{\sigma(i)}\).
Moreover,
\begin{align*}
    t(u_i)^2 &= x_i^2 v_{\s(i)}^2 = x_i^2 \sum_{s(g)=v_{\s(i)}} w'(g) \, r(g)
    = x_i^2 \sum_{s(f)=u_i} w'(\theta_\vec(f)) \, r(\theta_\vec(f)) \notag \\
    &= \sum_{s(f)=u_i} x_i^2 w'(\theta_\vec(f)) \, \theta_\vec(r(f))
    = \sum_{s(f)=u_i} w(f) \, x_j \, \theta_\vec(r(f)).
\end{align*}

Since $s(f)=u_i$ and $r(f)=u_j$, we have that  $\theta_\vec(r(f))=\theta_\vec(u_j)=v_{\s(j)}=x_j^{-1}t(u_j)$ and so 
\begin{align*}
 t(u_i)^2  & =\sum_{s(f)=u_i}w(f)x_jx_j^{-1}t(u_j)= 
\sum_{s(f)=u_i}w(f)t(r(f)) \\   & = t\left(\sum_{s(f)=u_i}w(f)r(f)\right)=t(u_i^2).
\end{align*}

The uniqueness of $t$ follows easily.
\end{proof}

 Let  $\mathcal{G}\colon\graph_K\to\E_K$ be the functor such that $\mathcal{G}(E,w)=(A,B)$, 
where $A$ is the $K$-algebra with natural basis $B=E^0=\{u_i\}_{i \in \Lambda}$ and product given by
$u^2=\sum_{f\in s^{-1}(u)} w(f)r(f)$, where $f\in s^{-1}(u)$.
Furthermore, for $\theta_\vec\colon (E,w)\to (E',w')$ we define 
$\mathcal{G}(\theta_\vec)$ as the homomorphism $t\colon (A,B)\to (A',B')$ such that $t(u_i)=x_i v_{\s(i)}$, where $B'=\{v_i\}_{i \in \Lambda'}$ is the natural basis of $A'$, $(x_i)_{i\in\Lambda}\in V(I_{\theta_\vec})$, and $\s\colon\Lambda\to\Lambda'$ is the map satisfying 
$\theta_\vec(u_i)=v_{\s(i)}$.

\begin{proposition}\label{weenoh}
  The categories $\E_K$ and $\graph_K$ are isomorphic in the sense
that $\mathcal{F}\mathcal{G}=1_{\graph_K}$ and $\mathcal{G}\mathcal{F}=1_{\E_K}$. 
In particular, the functor $\mathcal{F}$ induces a group isomorphism $\aut_{\E_K}(A,B)\cong\aut_{\graph_K}(E,w)$.
\end{proposition}
\begin{proof}
It is easy to check that $\mathcal{F}$ and $\mathcal{G}$ are mutually inverse functors. Now,
    if $\mathbf{C}_i$ are categories ($i=1,2$) and there are functors $\mathcal{A}\colon\mathbf{C_1}\to\mathbf{C_2}$ and $\mathcal{B}\colon\mathbf{C_2}\to\mathbf{C_1}$ such that $\mathcal{A}\mathcal{B}=1_{\mathbf{C_2}}$ and $\mathcal{B}\mathcal{A}=1_{\mathbf{C_1}}$, then $\mathcal{A}$ 
    maps monomorphicaly automorphisms of any object $U\in\mathbf{C_1}$ onto automorphisms
    of $\mathcal{A}(U)$ in $\mathbf{C_2}$. So it induces a group isomorphism $\aut_{\mathbf{C_1}}(U)\cong \aut_{\mathbf{C_2}}(\mathcal{A}(U))$. Applying this reasoning to the functors $\mathcal{F}\colon \E_K\to\graph_K$ and
    $\mathcal{G}\colon \graph_K\to \E_K$, we obtain the desired result.
    \end{proof}
    
\begin{example}\label{ex_aut_graphk}
Let $K$ be a field with $\car(K)\ne 2$ and consider the $2$-dimensional evolution algebra $A$ with natural basis $B=\{u_1,u_2\}$ and multiplication  $u_1^2=u_1+u_2$, $u_2^2=2u_1+u_2$. The associated weighted graph $(E,w)$ is

\begin{equation*}
    E: \begin{tikzcd}
\bullet  u_1 \arrow[r, bend left,"f"] \arrow[loop, distance=2em, in=215, out=145, "h"'] & \bullet u_2 \arrow[l, bend left,"g"] \arrow[loop, distance=2em, in=35, out=325,"k"']
\end{tikzcd}
\end{equation*}
where $w(h)=w(k)=1$, $w(f)=1$, and $w(g)=2$. We define the order two automorphism $\theta$ of $E$  given by

\begin{equation}\label{churros}
    \begin{array}{ll}
      \theta(u_1)=u_2, & \theta(h)=k,\\ 
       \theta(f)=g. &
    \end{array}
\end{equation}

Since $\aut_{\E_K}(A,B)\cong\aut_{\graph_K}(E,w)$ and $\aut_{\graph}(E)=\{1_E,\theta\}$. To compute $\aut_{\graph_K}(E,w)$, we only need
to verify if $\theta$ is a homomorphism in $\graph_K$. We compute $I_\theta$: it is the ideal of $K[x_1,x_2,y_1,y_2]$ generated by $x_iy_i-1$ ($i=1,2$), $w(h)x_1-w(k)x_1^2$,
$w(k)x_2-w(h)x_2^2$, $w(f)x_2-w(g)x_1^2$, and $w(g)x_1-w(f)x_2^2$. So,
$$I_\theta=(x_1y_1-1,x_2y_2-1, x_1-x_1^2,x_2-x_2^2,x_2-2x_1^2,2x_1-x_2^2).$$ It is easy to check
that $1\in I_\theta$, so $V(I_\theta)=\emptyset$ and consequently $\theta$ is not a homomorphism in the category $\graph_K$. This implies that \(\aut_{\graph_K}(E, w) = \{1\}\), and consequently, \(\aut(A, B) = \{1\}\). From this, given that $A$ is a perfect evolution algebra, the reader can also deduce that \(\aut_K(A) = \{1\}\), where automorphisms are considered in the category of \(K\)-algebras.

\end{example}

\begin{example}  
Consider again the graph \(E\) described in Example \ref{ex_aut_graphk}, now with a general weight \(w\). Let \(\theta\) be the automorphism in the category \(\graph\) defined by \eqref{churros}. We aim to determine the conditions under which \(\theta\) is a homomorphism in the category \(\graph_K\). To verify whether 
$\theta\in\aut_{\graph_K}(E,w)$, we must check if the following system has solutions $x_i\ne 0$, ($i=1,2$):
$$w(h)x_1=w(k)x_1^2,\ w(k)x_2=w(h)x_2^2,\ w(f)x_1=w(g)x_2^2,\ w(g)x_2=w(f)x_1^2.$$
It is not easy to find all the solutions of the above system, but it is straightforward to verify that if the field contains a cubic root \(\rho\) of \(1\), then a solution of the above system is 
 $x_1 = \rho, \, x_2 = \rho^2 \iff \big(w(f) = w(g), \, w(h) = \rho w(k)\big)  $. Thus, we have 
$$\begin{cases} \text{ if }w(h) = \rho w(k) \text{ and } w(f) = w(g), \text{ then } \aut_{\graph_K}(E, w) \cong \mathbb{Z}_2,\\ \text{ otherwise, } \aut_{\graph_K}(E, w) = \{1\}.\end{cases}$$

\end{example}

\subsection{On a class of automorphisms of evolution algebras}
The group isomorphism $\aut_{\graph_K
}(E,w)\cong \aut_{\E_K}(A,B)$ given in Proposition~\ref{weenoh}, followed by the group monomorphism $\aut_{\E_K}(A,B)\hookrightarrow\aut(A)$, provides the group monomorphism
$\aut_{\graph_K
}(E,w)\to\aut(A)$, which we denote by $\theta_{\vec}\mapsto\overline{\theta_{\vec}}$.

More generally, we have the following definition.

\begin{definition}\label{funtor}
    Let $\s\in\aut_{\graph_K}(E,w)$, and define the functor $\F_\s\colon\I_E\to\set$ as follows. For any $i\in E^0$, let $\F_\s(i)=K^\times$ and, for $a\in E^1$ with $s(a)=i$ and $r(a)=j$, define
$\F_\s(a)\colon K^\times\to K^\times$ by $x\mapsto k_{ji}x^2$, where $k_{ji}=\omega_{\s(j)\s(i)}/\omega_{ji}$. If $\l=a_1\cdots a_n$ is a path of length $n>1$, set
$\F_\s(\l)=\F_\s(a_1)\circ\cdots\circ \F_\s(a_n)$. Complete this definition by setting
$\F_\s(1_i)=1_{K^\times}$ for any $i$. 
\end{definition}

Following Remark~\ref{remark_lolitienesueño}, the inverse limit $\displaystyle 
\lim_{\leftarrow}\F_\s$ is the set of all $(x_i)_{i\in E^0}$, elements of $(K^\times)^{E^0}$, such that 
when $a$ is an edge in $E^1$ with source $i$ and range $j$ we have $\F_\s(a)(x_i)=x_j$. Equivalently,
$k_{ji}x_i^2=x_j$, which gives 
\begin{equation}\label{ident}
    \omega_{\s(j)\s(i)}x_i^2=x_j\omega_{ji}.
\end{equation} 
These conditions say that the map $A\to A$ such that $e_i\mapsto x_i e_{\s(i)}$ is an automorphism of $A$. 
Since the converse is clear, we have here a collection of injections 
\begin{equation}\label{Leo}   
\lim_{\leftarrow}\F_\s \hookrightarrow \aut(A),\quad (\s\in\aut\nolimits_{\graph_K}(E,w)).\end{equation}
Moreover, we define the operation $ 
    \bullet \colon\lim\limits_{\leftarrow}\F_\s \times \lim\limits_{\leftarrow}\F_\tau\to\lim\limits_{\leftarrow}\F_{\tau\s}$ 
    by  $$((x_i)_{i\in E^0},(y_i)_{i\in E^0})\mapsto  (x_i)\bullet(y_i):=(x_iy_{\sigma(i)})_{i\in E^0}.$$
    Observe that this operation is well-defined. Indeed, since $(x_i)_{i\in E^0}\in \L{\F_\s}$ and $(y_i)_{i\in E^0}\in \L{\F_\tau}$, this implies that 
    \begin{equation*}
\omega_{\tau\s(j)\tau\s(i)}x_i^2y_{\s(i)}^2=\omega_{\s(j)\s(i)}x_i^2y_{\s(j)}=\omega_{ji}x_jy_{\s(j)}
    \end{equation*} and so we have that $(x_iy_{\sigma(i)})_{i\in E^0} \in  \lim\limits_{\leftarrow}\F_{\tau\s}$. 
\begin{remark}\label{identificacion}
    When we view $\L{\F_\s}$ inside $\aut(A)$, 
    we can say that $\L{\F_\s}$ is the set of all automorphisms $\phi_\vec \colon A\to A$ such 
that $\phi_{\vec}(e_i)=x_i e_{\s(i)}$ for any $i$, where $\vec=(x_i)_{i\in E^0}$. Note that the conditions $\phi(e_i)=x_ie_{\s(i)}$ for all $i$ imply the identities \eqref{ident}. Furthermore, modulo the identification in \eqref{Leo}, the operation $\bullet$ corresponds to composition.
\end{remark} 

As a consequence, we have the following properties. 

\begin{proposition}  \label{umrof}
  Let $\s, \tau\in\aut_{\graph_K}(E,w)$. Then,
\noindent
 \begin{enumerate}[\rm (i)]
  
   \item $\L\F_1=\diag(A;B),$
   \item $(\L{\F_{\s}})^{-1}=\L{\F_{\s^{-1}}}$,
   \item $(\lim\limits_{\leftarrow}\F_\s)\circ(\lim\limits_{\leftarrow}\F_\tau)= \lim\limits_{\leftarrow}\F_{\s\tau}$,
\item $(\L{\F_\s})\circ(\L{\F_{\s^{-1}}})=\diag(A;B),$

\item  $\L{\F_\s}\cap\L{\F_\t}=\emptyset,\ (\s\ne\t),$
\item $ \diag(A;B) \circ \phi_\vec=\L{\F_\s}, \ \forall \phi_\vec\in\L{\F_\s}. $ 
\end{enumerate}
\end{proposition}

\begin{proof}
(i) is straightforward. For item (ii), let $\phi_\vec \in \L{\F_\s}$, so that $\phi_\vec(e_i)=x_ie_{\sigma(i)}$. Hence, $e_i=x_i \phi_\vec^{-1}(e_{\sigma(i)})$. This implies that $x_{\sigma^{-1}(j)}^{-1}e_{\sigma^{-1}(j)}=\phi_\vec^{-1}(e_j)$. Therefore, $\phi_\vec^{-1}\in \L{\F_{\sigma^{-1}}}$. Moreover, if $\phi_{\vec}\in \L{\F_{\s^{-1}}}$ then $\phi_{\vec}^{-1}\in (\L{\F_{{\sigma}^{-1}}})^{-1} \subset \L{\F_\s}$. Hence, $\phi_{\vec}\in (\L{\F_{\sigma}})^{-1}$. Item (iii) follows from item (ii) and Remark \ref{identificacion}.
Item (iv) is a consequence of items (i) and (iii).
Let us prove (v). If $\phi_\vec$ is an element in the intersection then $\phi_\vec(e_i)=x_i e_{\s(i)}=x_i e_{\tau(i)}$ for any $i$. If $\s(i)\ne\tau(i)$ then $e_{\s(i)}$ and $e_{\tau(i)}$ are linearly independent so
$x_i=0$, which is a contradiction.
Finally, to prove (vi), we first note that 
$\diag(A;B) \circ \phi_\vec=(\L{\F_1})\phi_\vec\subset \L{\F_1}\circ\L{\F_\s}\subset \L{\F_\s}$.
On the other hand, for any $\psi_\wec\in\L{\F_\s}$ we can write 
$\psi_\wec=(\psi_\wec \phi_\vec^{-1})\phi_\vec$ and 
$\psi_\wec \phi_\vec^{-1}\in\L{\F_\s}\circ\L{\F_{\s^{-1}}}=\diag(A;B).$ 
\end{proof}%
This suggests the following.

\begin{notation}\label{pescaito}  
Consider the union $\U:=\cup_\s \L\F_\s$ (where $\s$ ranges in the group $\aut_{\graph_K}(E,w)$). The formulas in \eqref{umrof} imply that $\U$ is a subgroup of $\aut(A)$ and that $\diag(A;B)$ is a normal subgroup of $\U$.
\end{notation}
As we will see in Corollary \ref{tesis}, for specific evolution algebras $A$, one has $\U=\aut(A)$. Notice that $\L\F_\s$ is not a group in general, but the union
of the various $\L\F_\s$ is. Now, we take into consideration  the map $p\colon\U\to\aut_{\graph_K}(E,w)$ defined as follows: for $t\in\U$, there is a permutation $\s$ of $\Lambda$ such that $t\in\L{\F_\s}$. So, there
is $(x_i)_{i\in\Lambda}\in K^\Lambda$ such that $t(u_i)=x_iu_{\s(i)}$. We define $\theta\in\aut_{\graph_K}(E,w)$ as in Lemma~\ref{previous}:       by writing $\theta(u_i)=u_{\s(i)}$ and, for an edge $a\in E^1\cap \s^{-1}(u_i)\cap r^{-1}(u_j)$, we define $\theta(a)$ as the unique edge in $E^1$ connecting $u_{\s(i)}$ to
$u_{\s(j)}$. It is straightforward to prove that $p$ is a group homomorphism whose kernel is $\diag(A;B)$. So, there is a short exact sequence of groups $$\diag(A;B)\hookrightarrow\U\buildrel{p}\over{\twoheadrightarrow} \aut\nolimits_{\graph_K}(E,w). $$
This sequence splits. Indeed, if we let $\iota\colon\aut\nolimits_{\graph_K}(E,w)\to\U$ be such that $\theta\mapsto t$, where $t$ is defined as $\mathcal{G}(\theta)$ (see Lemma \ref{erehiuqa}), then  $p\iota=1$ (the identity in $\aut\nolimits_{\graph_K}(E,w)$). Consequently, 
\begin{theorem}\label{ortauc}
With the notation in \ref{pescaito}, there is a group isomorphism $$\U\cong\diag(A;B)\rtimes\aut\nolimits_{\graph_K}(E,w).$$
\end{theorem}
\begin{remark} 
    Note that $\aut_{\E_K}(A,B)$ consists of the elements $t\in\aut(A)$ such that for any $b\in B$, there is some $b'\in B$ and $k\in K$ with $t(b)=kb'$. In view of Theorem \ref{ortauc},
   $\aut_{\E_K}(A,B)$ agrees with the 8group $\U$ defined in Notation \ref{pescaito}.
\end{remark}

\begin{corollary}\label{tesis}
  Let $A$ be an evolution algebra. If $A$ has  a natural basis $B$ satisfying the condition 2LI,  then $$\aut(A)\overset{(1)}{\cong}\diag(A;B)\rtimes\aut\nolimits_{\graph_K}(E,w)\overset{(2)}{\cong}\diag(A;B)\rtimes\aut{_{\E_K}}(A,B).$$

\noindent
In particular, this happens if the structure matrix $M_B(A)$ is invertible.
\end{corollary}
\begin{proof}
The isomorphism (1) is a consequence of Theorem \ref{2LI} and the isomorphism (2) of Proposition \ref{weenoh}.
\end{proof}

\subsection{Methods for inverse limit computations}

To illustrate how we can compute the inverse limit of $\F_\s$,  we will now consider several examples. The general idea is to consider, in the category of sets, the diagram obtained from the graph $E$ of the evolution $K$-algebra (relative to some natural basis $\{u_i\}$ such that 
$u_i^2=\sum_{j}\omega_{ji}u_j$) and replace each vertex by the set $K^\times$. Then, for every $a\in E^1$, we define $\F_\s(a)\colon K^\times\to K^\times$ by $\F_\s(a)(x)=k_{ji}x^2$, that is, $\F_\s(a)=k_{ji}s$, where $s$ is the squaring map and $k_{ji}=\omega_{\s(j)\s(i)}/\omega_{ji}$.

\begin{example}  
    Let $A$ be an evolution $K$-algebra with natural basis $\{u_i\}_{i=1}^3$ and product $u_1^2 = u_1+2 u_2 $, $u_2^2 = -u_2-u_3 $, and $u_3^2 = 2u_3-8u_1$. The graph associated to this evolution algebra is
\begin{center}
\hbox{
\begin{tikzcd}[row sep=1.5cm, column sep=1cm]
	& \bullet u_1 \\
	\bullet u_3 && \bullet u_2
	\arrow[from=1-2, to=1-2, loop, in=55, out=125, distance=12mm]
	\arrow[from=1-2, to=2-3]
	\arrow[from=2-1, to=1-2]
	\arrow[from=2-1, to=2-1, loop, in=145, out=215, distance=12mm]
	\arrow[from=2-3, to=2-1]
	\arrow[from=2-3, to=2-3, loop, in=325, out=35, distance=12mm]
\end{tikzcd}\hskip 1cm
$\begin{matrix}\omega_{11}=1, & \omega_{21}=2,\\
\omega_{22}=-1, & \omega_{32}=-1,\\
\omega_{13}=-8, & \omega_{33}=2.
\end{matrix}$
}
\end{center}
\noindent The group $\diag(A)$ is trivial (we do not specify the basis in the notation, since the algebra is perfect). On the other hand, the graph has cyclic symmetries of type $\s=(123)$ and $\sigma^2=(132)$. For $\s=(123)$ we have 
$$\begin{matrix}
k_{11}=-1 & k_{22}=-2 & k_{33}=1/2\\
k_{21}=-1/2 & k_{32}=8 & k_{13}=-1/4
\end{matrix}$$
The action of the functor $\F_\s$ is summarized in the diagram below.

\[\begin{tikzcd}[column sep = small]
	&& K^{\times}\\
	\\
	K^{\times} &&&& K^{\times}
	\arrow[from=1-3, to=1-3, loop, in=55, out=125, distance=10mm, "-s"]
	\arrow[from=1-3, to=3-5, "-\frac{1}{2}s"]
	\arrow[from=3-1, to=1-3, "-\frac{1}{4}s"]
	\arrow[from=3-1, to=3-1, loop, in=145, out=215, distance=10mm, "\frac{1}{2}s"]
	\arrow[from=3-5, to=3-1, " 8s"]
	\arrow[from=3-5, to=3-5, loop, in=325, out=35, distance=10mm, "-2s"]
\end{tikzcd}\]
 Now, we have to compute triples $(x_1,x_2,x_3)$
 such that $k_{j,i}s(x_i) = x_j$,
 That is:
  \begin{equation*}
      \begin{array}{cc}
          -x_1^2=x_1 ,& -\frac{1}{2}x_1^2=x_2, \\
           -2x_2^2=x_2, & 8x_2^2=x_3, \\
           \frac{1}{2}x_3^2=x_3, & -\frac{1}{4}x_3^2=x_1.
      \end{array}
  \end{equation*}
  
  Thus, the first three equations on the left side  give
 $x_1=-1$, $x_2=-1/2$, and $x_3=2$. These solutions are compatible with the remaining equations.  From this it follows that $\L\F_\s=\left \{(-1,-\frac{1}{2},2)\right \}$, which can be seen as an automorphism such that $u_1\mapsto -u_2$, $u_2\mapsto -\frac{1}{2}u_3$, and $u_3\mapsto 2u_1$. Using the matrix representation of the automorphisms and item (iii) of Proposition~\ref{umrof}, we can make the identifications  
 $$\L\F_\s\cong\left\{\begin{pmatrix*}[r]
 0 & 0  & 2\\
 -1 & 0 & 0\\
 0 & -1/2 & 0
\end{pmatrix*}\right\},
\quad 
\L\F_{\s^2}\cong\left\{
\begin{pmatrix*}[r]
0 & -1  & 0\\0 & 0 & -2\\1/2 & 0 & 0
\end{pmatrix*}\right\}.$$
Notice that $\L\F_1=\diag(A)=\{1\}$. Thus, $\aut(A)=\{1\}\sqcup\L\F_\s\sqcup\L\F_{\s^2}$.
 \end{example}

\begin{remark}  
    Since the algebra $A$ above is perfect and its associated graph is finite, $\diag(A)$ and $\aut(A)$  can also be computed using the methods presented in \cite{Elduque}. 
    As for the following example, this is not the case for infinite-dimensional algebras.
\end{remark}

\begin{example}  

Consider $A$ an evolution $K$-algebra with $B=\{u_i\}_{i\in \Z}$ a natural basis and product $u_i^2=u_{i}+u_{i+1}$ for all $i \in \Z$. The associated graph $E$ is described below.
\[E:\begin{tikzcd}
	{} & \underset{u_{-i}}{\bullet}& \underset{u_{-1}}{\bullet} & \underset{u_{0}}{\bullet} & \underset{u_{1}}{\bullet} & \underset{u_{i}}{\bullet} & {}
	\arrow[no head, dotted, from=1-1, to=1-2]
	\arrow[ from=1-2, to=1-2, loop, in=55, out=125, distance=10mm]
	\arrow[no head, dotted,from=1-2, to=1-3]
	\arrow[from=1-3, to=1-3, loop, in=55, out=125, distance=10mm]
	\arrow[from=1-3, to=1-4]
	\arrow[from=1-4, to=1-4, loop, in=55, out=125, distance=10mm]
	\arrow[from=1-4, to=1-5]
	\arrow[from=1-5, to=1-5, loop, in=55, out=125, distance=10mm]
	\arrow[no head, dotted, from=1-5, to=1-6]
	\arrow[  from=1-6, to=1-6, loop, in=55, out=125, distance=10mm]
	\arrow[no head, dotted, from=1-6, to=1-7]
\end{tikzcd}\]
 Let us prove that $A$ is not perfect. Fix $u_j$ of the natural basis. Then, there are scalars $k_{l}$,  only finitely many of which are nonzero, such that we can write
 $$u_j=\cdots+k_{j-2}u_{j-2}^2+k_{j-1}u_{j-1}^2+k_{j}u_{j}^2+k_{j+1}u_{j+1}^2+\cdots=$$
 $$\cdots+k_{j-2}(u_{j-2}+u_{j-1})+k_{j-1}(u_{j-1}+u_j)+k_{j}(u_{j}+u_{j+1})+k_{j+1}(u_{j+1}+u_{j+2})+\cdots .$$
 So, we get the system
 $$\begin{cases}
 \vdots \\
     0=k_{j-3}+k_{j-2},\\
     0=k_{j-2}+k_{j-1},\\
     1=k_{j-1}+k_{j},\\
     0=k_{j}+k_{j+1},\\
     0=k_{j+1}+k_{j+2},\\
   \vdots\\  
 \end{cases}$$
and hence, for $n\in \N^*$, we have $k_{j-n}=\pm k_{j-1}$ and $k_{j+n}=\pm k_{j}$. However, since
$k_{j-1}+k_j=1$,  there exist infinitely many nonzero scalars $k_{j-n}$ or $k_{j+n}$, which is a contradiction. Observe that this non-perfect evolution algebra can be obtained as the direct limit of a sequence of perfect evolution algebras. Thus,  perfection is not preserved under direct limits. Moreover, as we have seen in Section \ref{sec_2LI}, this algebra satisfies the \LI  \hspace{0.1pc} condition. 

In this case, all weights are one and $\diag(A;B)=\{1\}$. Moreover, the graph \( E \) exhibits translational symmetry: applying the shift \( u_i \mapsto u_{i+1} \) and iterating this operation for arbitrary lengths yields the same graph. 

Denote by $\s\colon E\to E$ the automorphism mapping each $u_i$ to $u_{i+1}$. 

\[\begin{tikzcd}
	{} & \underset{K^\times}{\bullet} & \underset{K^\times}{\bullet} & \underset{K^\times}{\bullet}& \underset{K^\times}{\bullet} & \underset{K^\times}{\bullet}& {}
	\arrow[no head, dotted, from=1-1, to=1-2 ]
	\arrow[ from=1-2, to=1-2, loop, in=55, out=125, distance=10mm, "s"]
	\arrow[no head, dotted,from=1-2, to=1-3 ]
	\arrow[from=1-3, to=1-3, loop, in=55, out=125, distance=10mm, "s"]
	\arrow[from=1-3, to=1-4, "s"]
	\arrow[from=1-4, to=1-4, loop, in=55, out=125, distance=10mm, "s"]
	\arrow[from=1-4, to=1-5, "s"]
	\arrow[from=1-5, to=1-5, loop, in=55, out=125, distance=10mm, "s"]
	\arrow[no head, dotted, from=1-5, to=1-6]
	\arrow[  from=1-6, to=1-6, loop, in=55, out=125, distance=10mm, "s"]
	\arrow[no head, dotted, from=1-6, to=1-7]
\end{tikzcd}\]

It is straightforward to check that  $\L\F_\s=\{(x_i)_{i\in\Z}: x_i=1   \text{ for any } i\}$. Moreover, the automorphism represented by the only element in $\L\F_\s$ is given by $f(u_i)=u_{i+1}$ ($i\in\Z)$.
In this case, $$\aut(A,B)=\bigcup_{n\in\Z
}\L\F_{\s^n}\cong \Z.$$

The conclusion is that by Theorem \ref{ortauc}, up to isomorphism, $\aut(A)\supset\Z$ because $\aut(A)\supset\U\cong\Z$. 

\end{example}

\begin{remark}  
    The example above shows that perfection is not preserved under direct limits, as the algebra $A$ can be written as the direct limit of the algebras $A_n$, where $A_n$ is the perfect evolution algebra with basis $\{u_{-n}, \ldots u_n\}$ and product $u_i^2=u_i+u_{i+1}, i=-n,\ldots,n-1$, and $u_n^2 = u_n$.
\end{remark}

\begin{remark}\label{hoy}
  
   Let $A$ be an evolution algebra with natural basis $B$. If $u\in B$ is the basis of a loop and $T(u)$ its tree, then for any $v\in T(u)$ and $t\in\diag(A;B)$, we have $t(v)=v$.

\end{remark}

\begin{example}
Consider the complex evolution algebra \( A \) with natural basis \( B = \{u_i\}_{i \in \mathbb{N}} \) and multiplication given by  \[
u_0^2 = u_0 \quad \text{and} \quad u_i^2 = \lambda_i u_i + \mu_i u_0 \quad \text{for } i > 0,
\]
where \( \{\lambda_i\} \) and \( \{\mu_i\} \) are two sequences of nonzero complex numbers.

The graph associated with \( A \) (relative to the basis \( B \)) is the following.

\begin{center}

     \begin{tikzpicture}[shorten <=2pt,shorten >=2pt,>=latex, node distance={15mm}, main/.style = {draw, fill, circle, inner sep = 1pt}]

\def \n {6}
\def \radius {1.5cm}
\def \margin {4} 

\node[main] (12) at (0,0) {};
\foreach \s in {0,...,4}
{
  \node[main] (\s) at ({360/\n * \s}:\radius) { };
  \draw[->] (\s) -- (12);
}

\draw[->,loop above] (0) to (0);
\draw[->,loop above] (1) to (1);
\draw[->,loop above] (2) to (2);
\draw[->,loop above] (3) to (3);
\draw[->,loop below] (4) to (4);
\draw[->] (12) to [out= 270, in = 300, looseness=50] (12);

  \draw[dotted] ({360/\n * (5 - 1)+\margin}:\radius) 
    arc ({360/\n * (5 - 1)+\margin}:{360/\n * (\n)-\margin}:\radius);
\def \Radius {1.85cm}
\node (6) at ({360/6 * 0-10}:\Radius) {$u_{1}$};
\node (7) at ({360/6 * 1-10}:\Radius) {$u_{2}$};
\node (8) at ({360/6 * 2+10}:\Radius) {$u_{3}$};
\node (9) at ({360/6 * 3-8}:\Radius) {$u_{4}$};
\node (10) at ({360/6 * 4-10}:\Radius) {$u_{5}$};

\node (15) at (0,0.5) {$u_0$};
\end{tikzpicture}   
\end{center}

The structure matrix of $A$ relative to $B$ is 
\begin{equation*}
    \left (\begin{array}{ccccc}
        1 & \mu_1& \mu_2 & \mu_3 & \cdots \\
        0& \lambda_1 & 0 & 0 & \cdots \\
        0& 0 & \lambda_2 & 0 & \cdots \\
        0& 0 & 0 & \lambda_3 & \cdots \\
        \vdots & \vdots & \vdots & \vdots & \ddots 
    \end{array}\right ).
\end{equation*}

It is easy to check that this matrix is invertible, with inverse 

\begin{equation*}
    \left (\begin{array}{ccccc}
        1 & -\mu_1\lambda_1^{-1} & -\mu_2\lambda_2^{-1} & -\mu_3\lambda_3^{-1} & \cdots \\
        0& \lambda_1^{-1} & 0 & 0 & \cdots \\
        0 & 0 & \lambda_2^{-1} & 0 & \cdots \\
        0& 0 & 0 & \lambda_3^{-1} & \cdots \\
        \vdots & \vdots & \vdots & \vdots & \ddots 
    \end{array}\right ).
\end{equation*}
\noindent
With only this data, Remark~\ref{hoy} ensures that \(\diag(A) = \{1\}\). The question is: What other automorphisms does $\aut(A)$ have? Notice that any permutation of $\N$ fixing $0$ induces a symmetry of the graph. Let $\Sigma$ be the set of all such permutations, and let $\sigma \in \Sigma$. We have $\omega_{00}=1$,
$\omega_{0i}=\mu_i$, $\omega_{ii}=\lambda_i$ for any $i>0$, and the other structure constants are null.
So, for $i>0$ we have $k_{ii}:=\omega_{\s(i)\s(i)}/\omega_{ii}=\lambda_{\s(i)}/\lambda_i$,
while $k_{0i}=\omega_{0\s(i)}/\omega_{0i}=\mu_{\s(i)}/\mu_i$. Also $k_{00}=1$. If we focus on the arrow $a$ from $u_i$ to $u_0$ (again $i>0$), this induces
$k_{0i}s\colon K^\times\to K^\times$. The loop at $u_i$ gives $k_{ii}s\colon K^\times\to K^\times$, and the loop at $u_0$ gives the squaring map $s\colon K^\times\to K^\times$. So, diagrammatically, we have 

\[\begin{tikzcd}
	{\underset{K^\times}{\bullet}} & {\underset{K^\times}{\bullet}}
	\arrow[from=1-1, to=1-1, loop, in=145, out=215, distance=10mm,"k_{ii}s"]
	\arrow[from=1-1, to=1-2,"k_{0i}s"]
	\arrow[from=1-2, to=1-2, loop, in=325, out=35, distance=10mm, "s"]
\end{tikzcd}\]
and $\L\F_\s$ consists on those sequence $(x_i)_{i\geq 0}$
such that $k_{ii}x_i^2=x_i$, $k_{0i}x_i^2=x_0$, and $x_0^2=x_0$. Then, 
$x_0=1$ and $x_i=\lambda_i/\lambda_{\s(i)}=\sqrt{\mu_i/\mu_{\s(i)}}$. So, we must have 
\begin{equation}\label{laidar}
\lambda_i^2/\lambda_{\s(i)}^2=\mu_i/\mu_{\s(i)}.
\end{equation} 
If this condition is satisfied by a permutation $\s\in\Sigma$, then $\L\F_{\s}$ is the set whose unique element is the automorphism $t\colon A\to A$ such
that $t(u_0)=u_0$ and $t(u_i)=x_i u_{\s(i)}$, where $x_i=\lambda_i/\lambda_{\s(i)}$.
On the other hand, note that $\Sigma'$, the set of permutations $\s\in \Sigma$ satisfying \eqref{laidar}, is a group. Since the structure matrix of $A$ relative to $B$ is invertible, we have 
$\aut(A)=\sqcup_{\sigma \in \Sigma'}\L\F_\s$, which is isomorphic to the group of permutations fixing $0$ and satisfying \eqref{laidar}. If there is no permutation satisfying \eqref{laidar}, then $\aut(A)=\{1\}$. If, for instance, we have the multiplication table 
$$\begin{cases} u_0^2=u_0, & \\  u_i^2=\alpha \ 2^{i/3}u_i+\beta\ 4^{i/3}u_0, & (i>0),
\end{cases}$$
for some nonzero constants $\alpha$ and $\beta$, then
$\lambda_i=\alpha \ 2^{i/3}$ and $\mu_i=\beta \ 4^{i/3}$ and it is easy to see that $\mu_i=k \lambda_i^2$ for some nonzero $k$. So \eqref{laidar} is satisfied for any $\s$
and the group of automorphisms of $A$ is isomorphic to the group of permutations of $\N\setminus\{0\}$.

\end{example}

\section*{Acknowledgements} 
The authors are supported by the Spanish Ministerio de Ciencia, Innovaci\'on y Universidades through project  PID2023-152673NB-I00. The first, fourth, fifth and sixth authors are supported by the Junta de Andaluc\'{\i}a  through project FQM-336. These two projects  with FEDER funds. The second and third author were partially supported by Funda\c{c}\~ao de Amparo \`a Pesquisa e Inova\c{c}\~ao do Estado de Santa Catarina (Fapesc) - Brazil. Conselho Nacional de Desenvolvimento Cient\'ifico e Tecnol\'ogico (CNPq) - Brazil partially supported the third author. The fourth and fifth authors are partially funded by grant Fortalece 2023/03 of "Comunidad Autónoma de La Rioja". 
The fourth and fifth authors were supported by the Brazilian Federal Agency for Support and Evaluation of Graduate Education â Capes, with Process numbers: 88887.895676/2023-00 and 88887.895611/2023-00, respectively. The sixth author is supported by the Junta de Andalucía PID fellowship no. PREDOC\_00029.

Commutative diagrams were drawn using the tikzcd editor
\url{https://tikzcd.yichuanshen.de/} of Y. Shen

\bibliographystyle{acm}
\bibliography{ref}

\end{document}